\documentclass[12pt,reqno]{amsart}
\usepackage[T1]{fontenc}
\usepackage[utf8]{inputenc}
\usepackage[english]{babel}
\usepackage{amsmath}
\usepackage{amssymb}
\usepackage{amscd}
\usepackage{amsfonts}
\usepackage{amsthm}
\usepackage{bm}
\usepackage{cite}
\usepackage{braket}
\usepackage{graphicx}
\usepackage{marginnote}
\usepackage[margin=1.3in]{geometry}
\usepackage[colorlinks=true, pdfstartview=FitV, linkcolor=darkblue, citecolor=darkred, urlcolor=cyan, linktoc=all]{hyperref}

\usepackage{dsfont} 
\usepackage{pxfonts}

\usepackage{color}
\usepackage{csquotes}

\definecolor{darkred}{RGB}{203,65,84}
\definecolor{darkblue}{RGB}{70,130,180}
\definecolor{brown}{RGB}{139,69,19}

\theoremstyle{plain}
\newtheorem{proposition}{Proposition}[section]
\newtheorem{theorem}[proposition]{Theorem}
\newtheorem{lemma}[proposition]{Lemma}
\newtheorem{corollary}[proposition]{Corollary}

\newcounter{foo}
\newtheorem{theo}[foo]{Theorem}

\theoremstyle{definition}
\newtheorem{definition}[proposition]{Definition}
\theoremstyle{remark}
\newtheorem{remark}[proposition]{Remark}

\numberwithin{equation}{section}

\DeclareMathOperator{\dal}{\Box}

\newcommand{\mk}{\mathfrak}
\newcommand{\bC}{\mathbf{C}}

\newcommand{\bF}{\mathbf{F}}

\newcommand{\mK}{\mathsf{K}}
\newcommand{\G}{\mathsf G}

\newcommand{\bM}{\mathbf M}
\newcommand{\mG}{\mathsf{G}}

\newcommand{\Proj}{\mathbf{P}}

\newcommand{\RR}{\mathbb{R}}

\newcommand{\cH}{\mathcal{H}}
\newcommand{\cM}{\mathcal{M}}
\newcommand{\cP}{\mathcal{P}}
\newcommand{\cG}{\mathcal{G}}

\newcommand{\cF}{\mathcal{F}}

\newcommand{\II}{\mathrm{II}}
\newcommand{\I}{\mathrm{I}}
\newcommand{\III}{\mathrm{III}}
\newcommand{\tr}{\operatorname{tr}}
\newcommand{\bq}{\mathbf{q}}

\newcommand{\cQ}{\mathcal{Q}}
\newcommand{\psld}{\mathsf{PSL}_2(\mathbb R)}

\newcommand{\dS}{\mathbf{dS}^{1,1}}
\newcommand{\Ker}{\mathrm{Ker}}

\newcommand{\D}{{\mathrm D}}
\newcommand{\q}{\mathbf{q}}

\renewcommand{\epsilon}{\varepsilon}

\newcommand{\PO}{\mathsf{PO}}
\newcommand{\SO}{\mathsf{SO}}

\newcommand{\PSL}{\mathsf{PSL}}

\newcommand{\R}{{\mathbb R}}
\newcommand{\sld}{\mathsf{SL}_2(\mathbb R)}
\newcommand{\Rp}{\mathbf{RP}^1}
\newcommand{\Rpn}{\mathbf{RP}}
\newcommand{\AdS}{\mathbf{H}^{2,1}}

\newcommand{\Ein}{\mathbf{Ein}^{1,1}}

\newcommand{\End}{\operatorname{End}}

\newcommand{\Id}{\operatorname{Id}}
\newcommand{\rI}{\mathrm{I}}

\newcommand{\Vol}{\mathbf{Vol}}

\newcommand{\Conf}{\operatorname{Conf}}

\newcommand{\B}{\mathrm{B}}

\newcommand{\Ad}{\operatorname{Ad}}

\newcommand{\vol}{\operatorname{vol}}
\newcommand{\function}[5]{#1 \left\{
    \begin{array}{rcl}
      #2 &\rightarrow& #3\\
      #4 &\mapsto& #5
    \end{array}
  \right.  }
\renewcommand{\leq}{\leqslant}
\renewcommand{\geq}{\geqslant}

\newcommand{\T}{\mathsf T}
\newcommand{\ms}{\mathsf}

\newcommand{\defeq}{\coloneqq}
\newcommand{\eqdef}{\eqqcolon}

\renewcommand{\d}{{\rm d}}

\newcommand{\U}{{\mathsf U}}

\newcommand{\cW}{{\mathcal W}}
\newcommand{\cL}{{\mathcal L}}

\newcommand{\sbt}{\,\begin{picture}(-1,1)(-1,-1)\circle*{2}\end{picture}\ }

\newcommand{\bb}{{\mathrm b}}
\newcommand{\bA}{\mathbf A}

\newcommand{\cS}{\mathcal{S}}

\newcommand{\dx}{\partial_x}
\newcommand{\dy}{\partial_y}
\renewcommand{\dot}[1]{\overset{\sbt}{#1}}
\title[Liouville action]
{Lorentz--Epstein surfaces and a Liouville action for positive curves}
\author[F. Labourie]{Fran\c cois Labourie}
\address{EPF Lausanne, SB-SCI-FL, Station 8, CH-1015 Lausanne,  Switzerland}
\email{francois.labourie@epfl.ch}

\author[J. Toulisse]{J\'{e}r\'{e}my Toulisse}
\address{Universit\'e C\^ote d'Azur, CNRS,  LJAD,  France}
\email{jtoulisse@univ-cotedazur.fr}

\author[Y. Wang]{Yilin Wang}
\address{ETH Z\"urich, Department of Mathematics, Switzerland}
\email{yilin.wang@math.ethz.ch}
\thanks{F.~L.  and J.~T. acknowledge funding by the European Research Council under ERC-Advanced grant 101095722. J.~T. acknowledges the support of the Institut Universitaire de France.  Y.~W. is supported by the Swiss State Secretariat for Education, Research and Innovation (SERI): MB25.00004. F.~L. is supported by the Swiss State Secretariat for Education, Research and Innovation (SERI): MB25.00031.}

\newcommand{\VB}{\operatorname{VB}}

\date{\today}
\AtBeginDocument{%
   \def\MR#1{}
}

\begin{document}
\begin{abstract}
We investigate and define in this paper, in the context of the correspondence between anti-de Sitter $3$-space and $(1,1)$-conformal metrics, the analogs of $\cW$-volume, Epstein surfaces, and Liouville action.  These notions were well-studied in the correspondence between $3d$-hyperbolic manifolds and $2d$ conformal metrics. We apply our construction to positive curves in flag manifolds equipped with a positive structure to obtain invariants of these curves that are finite in the case of {\em piecewise circles}.
\end{abstract}

\maketitle

\section{Introduction}

Renormalized volume is motivated by the holographic principle of AdS/CFT correspondence in String Theory discussed by C. Robin Graham and Edward Witten as well as Kirill Krasnov in  \cite{Graham_witten,krasnov2000holography}, allowing one to renormalize the volume of an infinite volume Einstein manifold using a truncation procedure determined by a metric on the conformal boundary. 

In the mathematics literature, the most studied case is the correspondence between convex co-compact hyperbolic 3-manifolds $P$, with conformal boundary $\partial_\infty P$ consisting of Riemann surfaces. The renormalized volume is defined as the $\cW$-volume of a submanifold $P_g$ obtained by ``truncating'' $P$ using a surface introduced by Charles Epstein \cite{epstein-envelopes} determined by a choice of the conformal metric $g$ on $\partial_\infty P$ \cite{KrasnovSchlenker_CMP,bridgeman2024epstein}:
\begin{align}
	\cW (P, g) := \vol (P_g)- \frac{1}{2} \int_{\partial P_g} H \ \d a \ ,\label{eq:cW-class}
\end{align}
where $\vol$ is the volume form of $P$, $H$ is the mean curvature, and $\d a$ the area form on $\partial P_g$ induced from $P$.

The case of $2+1$ dimensions was particularly interesting, as it was shown by Kirill Krasnov and Jean-Marc Schlenker in \cite{krasnov2000holography,KrasnovSchlenker_CMP} that the $\cW$-volume holographically expresses the Liouville action of Leon Takhtajan and Peter Zograf \cite{TZ87} on the conformal boundary. In particular, its Weyl anomaly follows Polyakov's formula, 
$$\cW (P,g) - \cW (P, e^{2\sigma} g) \propto \int_{\partial_\infty P} \left(\tfrac{1}{2} |\nabla \sigma|^2 + K_g \sigma \right) \ \d \vol_g \ , $$
where $K_g$ is the Gau\ss curvature of $g$. The right-hand side is proportional to the Liouville action $\cS_g(\sigma)$ with zero cosmological constant in the physics literature.

One of the first motivations for studying the Liouville action is the uniformization theorem. 
In fact, within the same conformal class of fixed area metrics, the critical point of the Liouville action is that associated with the constant curvature metric in $\partial_\infty P$ (i.e., the hyperbolic metric if the area is well-chosen).  Moreover, as a function on the Teichm\"uller space (i.e., when we vary the hyperbolic structure on $\partial_\infty P$), the $\cW$-volume turns out to be a K\"ahler potential for the Weil--Petersson metric (see Kirill Krasnov and Jean-Marc Schlenker as well as Leon Takhtajian and Lee-Peng Teo \cite{KrasnovSchlenker_CMP,Takhtajan:2003ur}) on Teichm\"uller space of the boundary. See the recent survey of Schlenker in \cite{schlenkerICM}. 

The $\cW$-volume can also be used to describe the geometry of Jordan curves on the Riemann sphere $\bf{CP}^1 = \partial_\infty \mathbf{H}^3$. In fact, each Jordan curve $\gamma$ determines two Epstein surfaces, meeting at $\gamma$, determined by the hyperbolic metrics on the connected components $\bf{CP}^1 \setminus \gamma$. The $\cW$-volume of the $3$-manifold between these two surfaces is proportional to the universal Liouville action (introduced by Leon Takhtajian and Lee-Peng Teo \cite{Takhtajan_Teo06}) of the curve studied by Martin Bridgeman, Kenneth Bromberg, Franco Vargas Pallete, and Yilin Wang in \cite{Bridgeman:2025aa}, which also has a deep link to the theory of random curves SLE by Yilin Wang \cite{Wang_equivalent,Wang_AMS}.
We also mention that the definition of Epstein (hyper)surfaces is not limited to $2+1$ dimensions. In fact, even in the simpler $1+1$ dimension, the ``renormalized area'' of the hyperbolic disk truncated by the Epstein curve coincides with the Schwarzian action (see Franco Vargas Pallete, Yilin Wang, and Catherine Wolfram in \cite{VWW}).

\bigskip

Nevertheless, the construction of Epstein hypersurfaces and renormalized volume is only studied in the Riemannian setup (i.e., in hyperbolic spaces, which are called Euclidean anti-de Sitter spaces by physicists).

\bigskip The goal of the present paper is to explore the definition of Epstein surfaces associated with a conformal metric of type $(1,1)$ and the corresponding $\cW$-volume and Liouville action for the Lorentzian anti-de Sitter space $\AdS$: while the hyperbolic 3-space is replaced by the $(2,1)$ anti-de Sitter space $\AdS$, we consequently replace $\bf{CP}^1$, with its conformal structure -- geometrically the boundary at infinity of the hyperbolic 3-space -- by the {\em Einstein Universe} $\Ein$, which is topologically a 2-dimensional torus and which has a conformal structure of type $(1,1)$. We will therefore proceed by analogy, as in the Riemannian case, to define Epstein surfaces and the $\cW$-volume in that context. Our construction gives rise to an invariant for positive curves, curves which are objects of interest by Jonas Beyrer, Olivier Guichard, François Labourie, Beatrice Pozzetti, and Anna Wienhard \cite{Beyrer:2024aa,Guichard:2025ab}, in particular for smooth hyperconvex curves in real projective spaces studied by Vladimir Fock and Alexander Goncharov as well as Fran\c{c}ois Labourie in \cite{Fock:2006a,Labourie:2006}. 

\bigskip

\noindent {\sc Epstein surfaces. } We first focus on the following situation. Let $(S,g)$ be a Lorentzian surface, then we define an {\em equivariant immersion} of $(S,g)$ in $\Ein$ as a pair $(\phi,\rho)$, 
\begin{enumerate}
	\item where $\phi$ is a conformal immersion from the universal cover of $(S,g)$ to $\Ein$,
	\item and  $\rho$ is  a representation of the fundamental group of $\pi_1(S)$ to $\psld\times\psld$ such that $\phi\circ\gamma=\rho(\gamma)\circ\phi$, for any element $\gamma$ in $\pi_1(S)$.
\end{enumerate} 
We remark that while this condition is satisfied locally, it does not need to  hold globally: as observed by Ravindra S. Kulkarni in \cite{Kulkarni:1985aa}, the global geometry of conformal Lorentz surfaces is considerably more subtle than their Riemannian analogue.

We first construct the Lorentzian analogs of Epstein surfaces. More precisely, we show in Theorem~\ref{theo:isotr-metr}:

\begin{theo} Let $(S,g)$ be a Lorentzian surface, and $\phi$ a conformal immersion from $(S,g)$ to $\Ein$ satisfying a topological hypothesis. Then there exists a  unique  holonomic surface $\Sigma_g$ in the space of tangent vectors of norm $1$ in $\AdS$, whose first fundamental form at infinity is $g$.
\end{theo}	
The terminology of this theorem requires some explanation and definitions that are given in the main part of the paper. For the moment, we just make the following remarks:
\begin{itemize}
	\item The topological hypothesis says that the pullback of a geometrically natural line bundle is trivial; this is made explicit in Theorem~\ref{theo:isotr-metr}. 
	\item We explain what a holonomic surface is in Section \ref{sec:holon}. For the sake of this introduction, we just say that a typical example of a holonomic surface is the set of normal vectors $n(S_g)$ to a surface $S_g$ of type $(1,1)$ in $\AdS$.
	\item In this example of typical holonomic surface, the first fundamental form at infinity is $\frac{1}{2}(\I +2\II +\III)$ where $\I$, $\II$ and $\III$ are respectively the first, second, and third fundamental forms of $S_g$ as we shall see in Proposition~\ref{pro:first-infinity}.
	\item These results extend immediately to equivariant conformal immersions thanks to the uniqueness part. 
\end{itemize}
The precise statement of the theorem gives more information on the uniqueness result.

We explain what this theorem means in terms of the envelope of AdS horospheres in paragraph \ref{sec:env}, recovering a classical feature of Epstein surfaces.  This result is the exact analog of the theorem giving Epstein surfaces in the Euclidean case \cite{epstein-envelopes}.

\bigskip

\noindent {\sc $\cW$-volume and 3-manifold in $\AdS$: }
Once Epstein surfaces are defined, we can proceed to the definition of the $\cW$-volume. 
 The $\cW$-volume is an invariant $\cW(M,\phi,\rho)$ for an equivariant map $(\phi,\rho)$ of a 3-manifold $M$ in  the space of unit spacelike vectors in $\AdS$, denoted $\U_+\AdS$ which means that 
\begin{enumerate}
	\item the map  $\phi$ is  a smooth map from the universal cover of $M$ to $\U_+\AdS$,
	\item and  $\rho$ is  a representation of the fundamental group of $\pi_1(M)$ to $\psld\times\psld$ such that $\phi\circ\gamma=\rho(\gamma)\circ\phi$, for any element $\gamma$ in $\pi_1(M)$.
\end{enumerate} 
We assume furthermore that 
\begin{enumerate}
	\item The map $\phi$ is an immersion when restricted to the lift of $\partial M$ to the universal cover of $M$,
	\item there is a compact set $K$ of $M$ such that $M\setminus K=S\times [0,1]$, where $S$ is a surface, and $\phi$ restricted to $S\times [0,1]$ does not depend on the second factor. 
\end{enumerate}  When $M$ is a product $M=N\times[0,1]$, this means in particular, up to some homéomorphism of $M$, that $\phi_0$ and $\phi_1$ from $N$ to $\U_+\AdS$ are "equal outside a compact set",  where $\phi_i(s)=\phi(s,i)$.

\bigskip 
In the case where the  boundary surface of $M$ is the set of normal vectors to an  (equivariantly) immersed  $(1,1)$ surface $\partial P$, hence bounding a 3-manifold $P$ smoothly mapped in $\AdS$ --- see paragraph \ref{sec:equal-at-infinity} --- we show in Proposition~\ref{pro:class-formula}

\begin{theo} The $\cW$-volume satisfies
	$$
\cW(M,\phi,\rho)=\Vol(P)- \frac{1}{2}\int_{\partial P}H\  \d a\ ,
$$
where $H$ is the mean curvature of $\partial M$ in $\AdS$ and $a$ its volume form.
\end{theo}

Observe the perfect parallel with equation \eqref{eq:cW-class} for the $\cW$-volume in the Hyperbolic context as proved by Kirill Krasnov and Jean-Marc Schlenker in  \cite{KrasnovSchlenker_CMP}.

We then prove the variational formula --- Theorem \ref{theo:VariationalFormula} --- for this $\cW$-volume, from which we draw two conclusions. 

\begin{enumerate}
	\item The $\cW$-volume only depends on the first fundamental form at infinity on $\partial M$. 
	\item Surfaces whose first fundamental form at infinity has constant curvature are critical points of this Liouville action with respect to compactly supported deformations: Corollary \ref{cor:critical_pt}.
	\end{enumerate} 
\bigskip
\noindent{\sc Liouville Action:}
We say two conformal Lorentzian metrics $g$ and $h$ on $S$ are {\em equal outside a compact set} when  $h=e^{2u} g$ with $u$ of compact support.  \begin{definition}
	Let $g$ and $h$ be two conformal metrics equal outside a compact set with $h=e^{2u}g$. The Liouville action of $(g,h)$ is then defined as
	\begin{align}
		\cS(g,h)= \frac{1}{2}\int_S u F_g - \frac{1}{4}\int_S u\ \d (\d u\circ \I)~,\label{def:Liouville-action-eq}
	\end{align}
	where   $F_g$ is the curvature form of $g$ and $\I$ is the {\em split involution} on $\T S$ whose eigenspaces define the light-like directions.
	\end{definition}
Our definition of Liouville action is proportional to the Lorentzian Liouville action in the physics literature with zero cosmological constant (or among metrics of the same area). See Remark~\ref{rem:liouville_physics} and, e.g., Joerg Teschner \cite{Teschner:2001ds}. 

It is also a perfect analog of the Liouville action defined by the Polyakov formula for conformal Riemannian metrics when one replaces the split involution $\I$ with the complex structure $J$.

The Liouville action is then shown to have the following properties in Proposition \ref{prop:liouville_expression}.

\begin{theo}\label{theo:B}
		Let $g$, $h$ and $k$ be metrics equal outside a compact set and let $u$ be of compact support  defined by $h=e^{2u}g$. Then
		\begin{align}
	\cS(g,h)&=\cS(g,k)+\cS(k,h)\ , \ \ \hbox{\sc(Chasles Formula)}\label{eq:intro-B1}\\
	\cS(g,h)&=\frac{1}{4}\int_S u(F_g+F_h)\ . \ \hbox{\sc(Monotonicity Formula)}
\end{align}
We have the {\em variational formula}: for a family of metrics $h_t=e^{2u_t}g$
	$$
	\left.\frac{\rm d}{\rm d t}\right\vert_{t=0}\cS(g,h_t)=\frac{1}{2}\int_S\dot u\  F_{g}\ ,
	$$ 
	where $\dot u$ is the function defined by $\dot u  =\left.\frac{\rm d}{\rm d t}\right\vert_{t=0}u_t$.	
	\end{theo}
	We also prove that constant curvature surfaces are exactly the critical points of the Liouville action for area-preserving deformations. More precisely, let us say that an {\em area-preserving deformation} of $g_0$ is a 1-parameter family of conformal metrics $(g_t)_{t\in\mathbb R}$ on $S$, equal outside a compact set $K$ satisfying 
	$$
	\left.\frac{\d}{\d t}\right\vert_{t=0}\vol_{g_t}(K)=0\ ,
	$$ 
As an immediate corollary of the variational formula

 \begin{theo}	If $(g_t)_{t\in\mathbb R}$ is a one-parameter family of area-preserving deformations such that $g_0$ has constant curvature, then  		
        $$
		\left.\frac{\d}{\d t}\right\vert_{t=0}\cS(g_0,g_t)=0\ .
	$$
	Conversely, if for any area-preserving variation of metrics, 
    $$
		\left.\frac{\d}{\d t}\right\vert_{t=0}\cS(g_0,g_t)=0\ ,
	$$
	then $g_0$ has constant curvature.
\end{theo}

Finally, we can relate the Liouville action to the $\cW$-volume in the following case. Assume 
\begin{enumerate}
	\item that $\{g_t\}_{t\in [0,1]}$ is a one-parameter family of conformal Lorentzian metrics on $S$,
	\item we have an equivariant immersion $\phi$ of $S$ in $\Ein$ for $t$ in $[0,1]$
\end{enumerate} 

Let then $\{(\phi_t,\rho)\}_{t\in [0,1]}$ the  family of equivariant immersions of $S$ in $\U_+\AdS$ which are  equivariant Lorentz--Epstein surfaces for $\phi$ and  
$\{\phi_t\}_{t\in[0,1]}$. Let finally $M=S\times[0,1]$ and $\psi:M\to\U_+\AdS$ where by abuse of language $\psi$ is a $\rho$-equivariant map defined by 
$$
\psi(s,t)=\phi_t(s)\ ,
$$ then  we prove in Theorem \ref{theo:cScW}:
  \begin{theo}
We have 
$$
\cS(g_0,g_1)=\cW(M,\psi,\rho)\ .
$$
\end{theo}
\bigskip 
\noindent{\sc $\cS$-class and Liouville action for the split Annulus: }
	So far, this discussion has been concerned with general conformal Lorentzian surfaces and associated metrics that coincide outside a compact set.  We now go beyond this discussion. 
    First, we restrict our discussion to the case of the split annulus $\bA$, conformal to $\dS$ -- see paragraph \ref{sec:dA}. We then extend it to a larger set of metric pairs.

    More precisely, we define an equivalence relation (Lemma~\ref{lem:Sclass-equi}) between metrics. The equivalence classes of this relation are called {\em $\cS$-classes}, see Definition \ref{def:Sclass}. We are then able to define the Liouville action in Definition~\ref{def:Liouville} for two $(1,1)$-metrics in the same $\cS$-class. The formula for the Liouville action for two metrics $g$ and $h$ related by the conformal factor $u$ such that $h=e^{2u}g$ is the same as formula  \eqref{def:Liouville-action-eq}
	-- see Lemma \ref{lemma:EquivalentSplit2} for details. With this $\cS$-class at hand, we prove in Propositions \ref{pro:propSclass} and \ref{pro:ActionBetweenUniformizing}:
	
	\begin{theo}\label{theo:F}
		The monotonicity formula and the Chasles relation hold for metrics in the same $\cS$-class. Moreover if two metrics $g$ and $h$ in the same $\cS$-class have the same constant curvature then 	\begin{align}\cS(g,h)&=0\ .\label{eq:intro-B4}\end{align} 
	\end{theo}

Observe that the last statement \eqref{eq:intro-B4} in  Theorem~\ref{theo:F} indeed makes sense: given a split annulus, there are several metrics of constant curvature in the same conformal class and in the same $\cS$-class (see Proposition \ref{pro:ActionBetweenUniformizing} for details). 
	
This theorem allows us to define the {\em Liouville action} for a metric $h$ in the $\cS$-class of a constant curvature metric $h_0$, as
\[ \cS(h)\defeq \cS(h_0,h)\ .\]
This definition is unambiguous in the choice of constant curvature metric $h_0$ in the conformal class of $h$, thanks to Chasles relation and assertion~\eqref{eq:intro-B4} in Theorem~\ref{theo:F}.

\bigskip
\noindent{\sc Positive curves:} We now apply these results for (locally) positive curves in a flag manifold $\cF$ associated with a group $\mG$ equipped with a positive structure. Positive structures were introduced in \cite{Guichard:2018vo} and positive curves in \cite{Guichard:2025ab}. Among classical examples are spacelike curves in the Einstein universe of arbitrary dimension; they may also arise as convex curves in the projective plane $\Rpn^2$ or more generally hyperconvex curves in $\Rpn^n$. They are discussed in section \ref{sec:positive-curves}.

We sketch now how positive curves give rise to $(1,1)$-metrics on $\bA$: if we consider a $C^1$ positive curve $c$ from $\Rp$ to $\cF$, we obtain a $C^1$-immersion from $\bA=\Rp\times\Rp\setminus \Delta$ to $\cG=\cF\times\cF\setminus \Delta$. The latter is equipped with a $(p,p)$ metric, and the induced metric $h_c$ on $\bA$ is $(1,1)$. 

We now restrict our discussion to {\em piecewise circles}: 
a special case of positive curves arising from {\em circles} which are orbits of {\em positive $\psld$} -- see Section \ref{sec:positive-flag}. The corresponding metric on $\bA$ has constant curvature $k$ (since $\psld$ is a transitive group of isometries), and the value of the constant curvature will depend on the choice of the conjugacy class of $\psld$.

We then define -- see Definition \ref{def:piecewise circle} -- a {\em piecewise circle} as a $C^1$-curve which is piecewise a circle for a given conjugacy class of $\psld$. This generalizes the notion of piecewise M\"obius curves discussed in Mario Bonk, Janne Junnila, Donald Marshall, Steffen Rohde, and Yilin Wang as well as Dragomir \v{S}ari\'{c}, Yilin Wang, and Catherine Wolfram \cite{v-Saric:2024aa,MRW2}. We can then compare the metric $h_c$ with the metric $h_0$ of constant curvature coming from a circle. We show the following result.

\begin{theo}
	The metric $h_c$ is in the $\cS$-class of $h_0$. Consequently, its Liouville action $\cS(h_c,h_0)$ is finite.
\end{theo}

Thus, given $\mG$, a positive structure on a flag manifold $\cF$ for $\mG$, and a positive $\psld$, 
we obtain an invariant (under the action of $\mG$) of a piecewise circle map $c$ in $\cF$ (with respect to our choice of $\psld$) as
$$
\cS(c)\defeq \cS(h_0,h_c)\ .
$$
Our main result is then that $\cS(c)$ is finite and circles are critical points of this action $\cS(c)$ by Corollary~\ref{cor:critical_pt}.

\tableofcontents

\section{Conformal Lorentz surfaces}\label{sec:conformal}

\subsection{Split structure}

Let $V$ be a two-dimensional real vector space. We fix an orientation on $V$. 

\subsubsection{Split vector spaces} A {\em split structure} on $V$ is a pair $(V_1,V_2)$ of distinct lines in $V$.  A {\em split basis} is then an oriented basis $(v_1,v_2)$ of $V$ with $v_i$ spanning $V_i$.  The {\em canonical involution} $\rI$ is the involution of $V$ preserving each $V_i$ and such that $\rI_{\vert V_i}= (-1)^i$.

 A {\em Lorentz product} is a quadratic form of signature $(1,1)$ on $V$. Split structures are linked with Lorentz metrics as follows.  For a Lorentz metric $\bq$,  there is a unique split structure $(V_1,V_2)$ on $V$ such that each $V_i$ is $\bq$-isotropic (namely, $\bq_{\vert V_i}=0$) and $\bq(v_1,v_2)$ is positive for any split basis $(v_1,v_2)$. The quadratic form $\bq$ is then said to be {\em compatible} with the split structure $(V_1,V_2)$.

Recall that an endomorphism $\varphi$ of $(V,\bq)$ is {\em conformal} if $\varphi^*\bq=e^\lambda \bq$ for some real number $\lambda$. The group $\Conf(V,\bq)$ of conformal endomorphisms of $V$ has four connected components, and we denote by $\Conf_+(V,\bq)$ the index 2 subgroup consisting of orientation-preserving conformal endomorphisms. Observe that $\Conf_+(V,\bq)$ is isomorphic to the nonconnected Lie group $\Conf_+(1,1)=\RR_{>0}\times \SO(1,1)$.

\begin{lemma}\label{lemma:EquivalentSplit1}
An orientation-preserving endomorphism of $(V,\bq)$ is conformal if and only if it preserves the associated split structure.
\end{lemma}

\begin{proof}
Let $\varphi$ be an orientation-preserving endomorphism of $V$ and let $(v_1,v_2)$ be a split basis.

If $\varphi$ is conformal, then it maps $v_i$ to an isotropic vector, so it globally preserves $V_1\cup V_2$. To prove that $\varphi$ cannot exchange $V_1$ and $V_2$, just observe that if $\varphi(v_i)=\mu_i v_{i+1}$ then the conformality of $\varphi$ implies $\mu_1\mu_2>0$ and so $\varphi$ reverses the  orientation.

Conversely, if $\varphi$ preserves each $V_i$, since it preserves orientation, there exists $\mu_i$ such that $\varphi(v_i)=\mu_iv_i$ with $\mu_1\mu_2>0$. So $\varphi$ is conformal.
\end{proof}

In particular, a split structure on $V$ is equivalent to a conformal class of Lorentz structures.

\subsubsection{Split surfaces and Lorentz metric} Let $S$ be a smooth, oriented surface. A {\em split structure} $\boldsymbol\sigma$   on $S$ is a pair $(\cL_1,\cL_2)$ of transverse 1-dimensional foliations. A {\em split surface} is then a pair $(S,\boldsymbol\sigma)$ where $S$ is an oriented surface and $\boldsymbol\sigma$ is a split structure on $S$. A {\em split framing} is a frame $(u_1,u_2)$ of $\T S$ such that at any point $x$, the pair $(u_1(x),u_2(x))$ is a split basis of $\T_xS$ with split structure $(\T_x\cL_1,\T_x\cL_2)$.

A ($C^k$) {\em Lorentz metric} on $S$ is a ($C^k$) field of Lorentz products on $\T S$. Two Lorentz metrics $g$ and $h$
 are {\em conformally equivalent} or {\em conformal} if there is a function $f$ on $S$ such that $g=e^{2f}h$. We define $\cM(S,\sigma)$ to be the space of metrics in the conformal defined by the split structure of $\sigma$.

Lemma~\ref{lemma:EquivalentSplit1} has the following consequence:

\begin{lemma}\label{lemma:EquivalentSplit2}
Given a smooth, oriented surface $S$, the following structures are equivalent:
\begin{enumerate}
\item a split structure,
\item a conformal class of Lorentz metric,
\item a field of involutions $\rI$ in $\Gamma(S,\End(\T S))$ with $1$-dimensional eigenspaces,
\item a reduction of the structure group of the bundle of oriented frames to $\Conf_+(1,1)$.
\end{enumerate}
\end{lemma}
\begin{remark}
In the equivalence described above, the leaves of the foliation $\cL_i$ are the integral curves of the distribution $\Ker(\rI-(-1)^i\Id)$ on $S$.
\end{remark}

Observe that the existence of a split structure on a surface $S$ implies that the tangent bundle of $S$ is trivial. In particular, if $S$ is closed, then it is diffeomorphic to the torus.

A {\em split map} between two split surfaces $(S_1,\boldsymbol\sigma_1)$ and $(S_2,\boldsymbol\sigma_2)$ is a homeomorphism $f$ from $S_1$ to $S_2$ that sends lightlike geodesics to lightlike geodesics. We will only be interested in split diffeomorphisms, which can be characterised as maps that are conformal with respect to the underlying Lorentz conformal structures. The horizontal and vertical lines define a standard split structure $\boldsymbol\sigma_0$ on $\mathbf R^2$ and an {\em isothermal coordinate} on $(S,\boldsymbol\sigma)$ is a local chart with values in $(\mathbf R^2,\boldsymbol\sigma_0)$ which is a split map. 

\begin{lemma}[\sc Existence of isothermal coordinates]
Let $(S,\boldsymbol\sigma)$ be a split surface. Then locally $(S,\boldsymbol\sigma)$ admits isothermal coordinates.
\end{lemma}

\begin{proof}
Let $(X_1,X_2)$ be a split framing around a point $p$. One can find positive functions $f_1$ and $f_2$ such that the new split framing $(Y_1,Y_2)$ with $Y_i=f_iX_i$ satisfies $[Y_1,Y_2]=0$. In particular, the flows commute, and the inverse of the map
\[F(s,t)= \Phi_{Y_1}^s \circ \Phi_{Y_2}^t (p)\]
defines isothermal coordinates around $p$.
\end{proof}

\subsubsection{Compatible Lorentz structure} Let $(S,\boldsymbol\sigma)$ be a split surface with canonical involution $\rI$, and let $V_i$ be the distribution tangent to $\cL_i$ for $i=1,2$.

\begin{definition} A Lorentz metric $g$ on $S$ is {\em compatible with $\boldsymbol\sigma$} if its conformal class coincides with $\boldsymbol\sigma$ (see Lemma \ref{lemma:EquivalentSplit2}). For $k \ge 1$, we denote by $\cM^k(S,\boldsymbol\sigma)$ the space of $C^k$-Lorentz metrics on $S$ compatible with $\boldsymbol\sigma$.
\end{definition}

\begin{remark} Observe that $C^k(S)$ acts on $\cM^k(S,\boldsymbol\sigma)$, where the action is given for a function $u$ and a metric $g$ by $(u,g)\mapsto e^{2u}g$. This action is simply transitive, since any two metrics in $\cM^k(S,\boldsymbol\sigma)$ are conformal. In other words, the space $\cM^k(S,\boldsymbol\sigma)$ is an affine space over  $C^k(S)$.
\end{remark}

\begin{remark}
For any Lorentz metric $g$ in $\cM^k(S,\boldsymbol\sigma)$, its {\em volume form} $\omega_g$  compatible with the orientation satisfies 
\[\omega_g(u,v)= g(u,\rI v)~.\]
Indeed, the standard flat Lorentz metric on $\R^2$ is given by $$g_{\scriptscriptstyle{flat}}=\d x\ \d y$$  where $\partial_x$ spans $V_1$ and $\partial_y$ spans $V_2$, and the corresponding volume form is $\omega_{\scriptscriptstyle{flat}} =  \d x\wedge \d y$. In particular, if $(v_1,v_2)$ is a split basis for $g$ satisfying $g(v_1,v_2)=1$, then
$$
\omega_g(v_1,v_2)=1\ . $$ 
It follows that the map 
\[\function{
\varphi:}{\cM(S,\boldsymbol\sigma)}{\Omega^2_+(S)}{g}{\omega_g} \]
defines a one-to-one correspondence between $\cM^k(S,\boldsymbol\sigma)$ and the set $\Vol_+^k(S)$ of $C^k$-volume forms compatible with the orientation.
\end{remark}

\begin{remark} Observe finally that if $\varphi$ is a $C^{k+1}$-split diffeomorphism from $(S_1,\boldsymbol\sigma_1)$ to $(S_2,\boldsymbol\sigma_2)$, then the pull-back $\varphi^*$ induces a one-to-one correspondence from $\cM^k(S_2,\boldsymbol\sigma_2)$ to $\cM^k(S_1,\boldsymbol\sigma_1)$.
	\end{remark}

\subsubsection{Curvature of a Lorentz surface}

Let $g$ be a compatible metric on a split surface $(S,\boldsymbol\sigma)$, and denote by $\nabla$ the associated Levi-Civita connection. The {\em d'Alembertian} $\dal_g$ is the differential operator defined on a  $C^2$ function $f$ by
\[\dal_g f = \tr_g(\nabla \d f)~.\]
We recall that, for a symmetric bilinear form $Q$ associated with the linear operator $A$ by $Q(u,v)=g(A u,v)$, $$
\tr_g(Q)\defeq  2Q(v_1,v_2)=\tr(A)\ 
$$
where $(v_1,v_2)$ is a split basis for $g$ satisfying $g(v_1,v_2)=1$. Observe in particular that $\tr_g(g)=2$.

Let $R_g$ be the curvature tensor of $g$, and let $\I$ define the split structure. The {\em sectional curvature} of $g$ is denoted by $K(g)$ and defined by the relation 
\[R_g = -K(g)\omega_g\otimes \I\ .\]
The {\em curvature $2$-form} of $g$  is  $F_g\defeq K(g) \omega_g$. Thus  $R_g=-F_g\otimes I$.

\begin{proposition}[\sc Conformal change]\label{pro:ConformalChange}
Let $g$ and $h$ be metrics in $\cM(S,\boldsymbol\sigma)$ with $h=e^{2u}g$. Then
\begin{align}
\dal_h&=e^{-2u}\dal_g\ ,\label{eq:daldal} \\ 
(\dal_g u)\omega_g&=\d(\d u\circ \I) = F_g-F_h~.
\end{align}

Equivalently, we have 
\[\dal_g u = K(g)-e^{2u} K(h)\]
which is the Lorentzian analog of the conformal change formula of the sectional curvature. 
\end{proposition}

To see this,  we first compute the relevant terms in isothermal coordinates.
\begin{lemma}\label{lem:isothermal}
Let $g$ be a $C^2$-Lorentz metric on $S$ given by $g=e^{2u}\d x \d y$ in some isothermal coordinates $(x,y)$. Let $f$ be a function on $S$.

\begin{enumerate} 
\item The Levi-Civita connection $\nabla$ of $g$ is given by
\begin{align*}	
	\nabla_{\partial_x}\partial_y  = \nabla_{\partial_y}\partial_x = 0\ \ , \ \
	\nabla_{\partial_x}\partial_x  = 2 (\partial_x u) \ \partial_x\ \ , \ \
	\nabla_{\partial_y}\partial_y  = 2 (\partial_y u)\ \partial_y \ .
\end{align*}
\item The d'Alembertian with respect to $g$ is given by
$$\dal_g f= 2e^{-2u}\partial^2_{xy}f~.$$
\item The sectional curvature of $g$ satisfies $$K(g)=-\dal_gu~.$$ 
\end{enumerate}
\end{lemma}

\begin{proof}
Differentiating the equations $g(\partial_x,\partial_x)=g(\partial_y,\partial_y)=0$, one obtains that both $V_1$ and $V_2$ are parallel with respect to $\nabla$. Thus, $\nabla$ splits into $\nabla=\nabla^1\oplus\nabla^2$ with $\nabla^i$ being a connection of the bundle $V_i$. Using $[\partial_x,\partial_y]=0$  we get
\[
\nabla_{\partial_x}\partial_y  =\nabla_{\partial_y}\partial_x\in V_1\cap V_2=\{0\} ~.\]
Differentiating $g(\partial_x,\partial_y)=e^{2u}$ then gives the expression of $\nabla$. This concludes the proof of item (1).

For item (2), given a function $f$, using the definition of the d'Alembertian, we have
\[\nabla \d f = \tfrac{1}{2}(\dal_g f)g + H_0~,\]
where $H_0$ is traceless and thus $H_0(\dx,\dy)=0$. This gives
\begin{align*}
\dal_g f & = \frac{2\nabla \d f (\partial_x,\partial_y)}{g(\partial_x,\partial_y)} \\
& = \frac{2\left( \partial_x(\d f(\partial_y))  - \d f(\nabla_{\partial_x}\partial_y)\right)}{e^{2u}}\\
& = 2e^{-2u} \partial^{2}_{xy}f~.
\end{align*}
Finally, for item (3), using the previous items, we have
\begin{align*}
R_g(\partial_x,\partial_y)\partial_y & =\nabla_{\partial_x}\nabla_{\partial_y}\partial_y \\
& = 2(\partial^2_{xy}u)\ \partial_y \\
& = (\dal_g u)  \omega_g(\partial_x,\partial_y) \I(\partial_y)~.
\end{align*}
The expression of $K(g)$ follows.
\end{proof}

\begin{proof}[Proof of Proposition \ref{pro:ConformalChange}] Observe first that item (2) of Lemma \ref{lem:isothermal} implies that
$$
\dal_h=e^{-2u}\dal_g\ .
$$ 
Write $g=e^{2v}\d x\d y$ in isothermal coordinates. Thus, the previous lemma gives
\[K(h)=-\dal_h(u+v) = -e^{-2u}\dal_g(u+v)\ .\] 
Thus 
\[K(g)-e^{2u}K(h) = -\dal_g v + \dal_g u + \dal_g v = \dal_gu~. \]
Now we have
\[\d (\d u \circ \I) = 2\partial^2_{xy}u\  \d x\wedge \d y = (\dal_g u)\omega_g~,\]
hence
$$F_g-F_h = \left(K(g) - e^{2u}K(h)\right) \omega_g = (\dal_g u)\omega_g =  \d (\d u\circ \I)~.$$
This concludes the proof. \end{proof}

\subsection{De Sitter surface}
\label{sec:dS}

Given a 3-dimensional real vector space $E$ equipped with a quadratic form $Q$ of signature $(2,1)$, the {\em de Sitter surface} is
\[\dS    = \{ x \in E~,~ Q(x)=1\}~.\]
The quadratic form $Q$ restricts to a Lorentz product on each tangent space $\T_x \dS   $ since this tangent space is identified with $x^\bot$. The resulting metric has constant curvature $+1$ and the group $\SO_0(Q)$ acts by isometries.

Given a point $x$ in $\dS$ and a split basis $(v_1,v_2)$ at $\T_x \dS$, the leaf $\cL_\alpha$ through $x$ is given by $x+\R v_\alpha$ and so its projection to $\Proj(E)$ intersects the quadric $\Rp\cong\{x\in \Proj(E)~,~\q_{\vert x}=0\}$ in $[v_\alpha]$. This defines a map
\[ \function{\Phi:}{\dS}{(\Rp\times\Rp)\setminus\Delta ,}
{p}{([v_1],[v_2])\ ,}\]
where $\Delta$ is the diagonal. Observe that given two distinct points $(x_1,x_2)$ in $(\Rp\times\Rp)\setminus\Delta$, the planes $x_1^\bot$ and $x_2^\bot$ intersect along a positive line and one can find a unique point $p$ in $\dS$ on this line such that $V_\alpha$ is contained in $x_\alpha$. Thus, $\Phi$ is bijective and is indeed a diffeomorphism. In this model, the split structure is given by the fiber of the projection on each factor. 

\begin{lemma}
Any affine chart on $\Rp$ defines an open set in $(\Rp\times\Rp)\setminus\Delta$ in which the de Sitter metric $g_0$ satisfies
\[g_0 = \frac{2\ \d x\ \d y}{(x-y)^2}~.\]
\end{lemma}

\begin{proof}
Any split metric on $(\R \times \R) \setminus \Delta$ invariant under the affine group acting diagonally on $\R \times \R$ is of the form $\frac{\lambda \d x\d y}{(x-y)^2}$ for some positive $\lambda$. The value of $\lambda$ is determined by the condition that the curvature is equal to $1$. 
\end{proof}

\subsection{Split annulus} 
\label{sec:dA}
\begin{definition}
The {\em split annulus} is the split surface $\bA$ underlying $\dS   $, that is, $\bA=(\Rp\times \Rp)\setminus \Delta$ where the split structure $(\cL_1,\cL_2)$ is given by the fibers of the projection on the $i^{th}$ factor.
\end{definition}

\begin{proposition}\label{pro:SplitDiff}
Every $C^k$-split map of $\bA$ is of the form
$$
\Phi:(x,y)\mapsto (\varphi(x),\varphi(y))\ ,
$$
where $\varphi$ is a $C^k$-diffeomorphism (or homeomorphism for $k=0$) of $\Rp$. The map $\Phi$ is an isometry of $(\bA,g_0)$ if and only if $\varphi$ is projective. 
	
\end{proposition}

\begin{proof}
	Let $\Phi$ be a split homeomorphism of $\bA$ to itself, namely, $\Phi$ sends each oriented foliation to itself. Let us write $\Phi(x,y)=(f_y(x),g_x(y))$. For every $y$, $f_y$ is a homeomorphism from $\Rp\setminus\{y\}$ to $\Rp\setminus\{z\}$, for some $z \in \Rp$ which we denote $z : = \varphi (y)$. We can thus extend $f_y$ to $\Rp$ as a homeomorphism by setting $f_y(y)=\varphi(y)$. Similarly for $g_x$. As $\Phi$ is a split homeomorphism, we observe that $g_x(y)$ does not depend on $x$ and $f_y(x)$ does not depend on $y$. It follows that $\Phi(x,y)=(f(x),g(y))$. Since $\Phi$ sends the diagonal to itself, we have $f (z) = g (z) = \varphi (z)$ for all $z \in \Rp$. We also obtain that $\varphi$ is a homeomorphism of $\Rp$ as $\Phi (x,y) = (\varphi(x), \varphi(y))$ is a homeomorphism of $\bA$. This completes the proof.
\end{proof}

Observe that, unlike the hyperbolic disk, the group of conformal diffeomorphisms of $\bA$ is infinite-dimensional and thus much larger than the finite-dimensional group of isometries.

\section{Isotropic surfaces}

In this section, we consider a four-dimensional vector space $W$ equipped with a quadratic form $\bq$ of signature $(2,2)$ and denote by $\langle\cdotp,\cdotp\rangle$ the corresponding polar form. 

\subsection{The split structure of the Einstein torus}\label{ss:Einstein}
The {\em 2-dimensional Einstein Universe} or {\em Einstein Torus} is the quadric
\[\Ein  = \{ x\in \Proj(W)~,~\bq(x)= 0 \}~.\]

Let $\tau$ be the {\em tautological line bundle} over the Einstein Torus $\Ein $, that is, the line bundle whose fiber $\tau_x$ over a point $x$ is the isotropic line defined by $x$ in $W$.

The Einstein Torus is naturally equipped with a canonical split structure $(\cL_1,\cL_2)$ which we now describe. Let $(V,\omega)$ be a 2-dimensional real vector space equipped with a volume form $\omega$, and consider $(V\otimes V,\bq)$, where $\bq\defeq -\omega\otimes\omega$. 
Explicitly, we have
\[-\omega\otimes\omega(u_1\otimes u_2,v_1\otimes v_2) = - \omega(u_1,v_1)\omega(u_2,v_2)~,\]
so $-\omega\otimes\omega$ is a signature $(2,2)$ quadratic form on the $4$-dimensional vector space $V\otimes V$. Once we choose an isomorphism between $(W,\bq)$ and $(V\otimes V,-\omega\otimes\omega)$, the image of the Segre embedding
\begin{equation}\label{eq:Segre}
\function{\mathcal S:}{\Proj(V) \times \Proj(V)}{\Proj(V\otimes V)\ ,}
{([v_1],[v_2])}{[v_1\otimes v_2]\ ,} 
\end{equation}
 is exactly $\Ein $. This defines an isomorphism between $\Ein $ and $\Rp \times \Rp$. Define $\cL_i$ as the foliation of $\Ein $ whose leaves are the fibers of the projection on the $i^{\operatorname{th}}$ factor.

\subsection{Isotropic surfaces} 

An {\em isotropic surface} is an immersion $\sigma$ of a surface $S$ in $W$ whose image of every point is an isotropic vector and such that $\sigma(s)$ is transverse to the image $I_s$ of  $\T_s\sigma$. 
Equivalently,  we have for any $s$ in $S$, for any $u$ in $\T_sS$, 
$$
 \braket{\sigma,\sigma}=0\ , \dim(\textrm{span}\{\D_u\sigma,\sigma(s)\})=2\ .
$$
Observe that projecting $\sigma$ to $\Proj(W)$ defines an immersion $[\sigma]$ from $S$ to $\Ein $.  Two isotropic surfaces $\sigma_0$ and $\sigma_1$ are equivalent if $\sigma_0=\pm\sigma_1$ .

Our first result is the following.
\begin{theorem}\label{theo:isotr-metr}
Let $\phi$ be an immersion of $S$ in $\Ein $, such that $L\defeq\phi^*\tau$ is trivializable over $S$.  Let $g$ be any $C^k$-Lorentz metric on $S$ compatible with the induced split structure.  Then there exists a unique equivalence class of $C^{k}$-isotropic surface $\sigma$ from $S$ to $W$, such that $[\sigma]=\phi$ and  $$
g(u,v)=\braket{\D_u\sigma, \D_v\sigma}\ .
$$
\end{theorem}
 We remark that although the formula defining $g$ seems at first sight to depend on the first derivative of $\sigma$, due to the "lightlike" nature of $\sigma$, $g$ only depends on $\sigma$ pointwise as we shall see in the proof.

\vskip 0.1 truecm
Let $\phi$ be an immersion of $S$ in $\Ein $.  Observe that the natural inclusion of $L$ in $W$ defines an immersion $i$ from the complement $L^*$ of the zero section in $L$ to $W$. We start with a remark: in the sequel,
we shall freely identify (equivalence classes of) isotropic surfaces with sections of $L$.

Now, the proof of the theorem follows from the following  lemma:

\begin{lemma}
	Assume that $L$ is trivial and let $L^+$ be a connected component of the complement of the zero section in $L$.  
	
	Then, the map which associates with a section $\sigma$ of $L^+$ the metric $g_\sigma$ defined by $g_\sigma(u,v)=\braket{\D_u\sigma,\D_v\sigma}$ is a bijection from the space $\Gamma^k(L^+)$ of $C^k$-sections to the space $\cM^k(S)$ of $C^k$-metrics on $S$ in the conformal class determined by the split structure on $S$.
\end{lemma}

\begin{proof}
Let $\sigma_0$ and $\sigma_1$ be two sections of $L^+$. Let us write $\sigma_0=e^f \sigma_1$. Then 
\begin{align}
	g_{\sigma_0}(u,v)=\braket{\D_u(e^f\sigma_1),\D_v(e^f\sigma_1)}=e^{2f}\braket{\D_u\sigma_1,\D_v\sigma_1}=e^{2f}g_{\sigma_1}(u,v)\ .\label{eq:g-sigma}
\end{align}

In the second equality, we used the fact that for any split surface $\sigma$, then $\braket{\sigma,\sigma}=0$. Hence, 
 after differentiation, for any vector $u$, $\braket{\D_u\sigma,\sigma}=0$. Finally, the lemma follows immediately from equation~\eqref{eq:g-sigma}.
\end{proof}

\subsection{Dual isotropic surfaces and forms at infinity} 
\label{ss:dual_isotropic}
Given an isotropic surface $\sigma$, we call a {\em dual isotropic surface} to $\sigma$ a map $\eta$ such that 
$$
 \braket{\eta,\sigma}=1\ , \  \braket{\eta,\eta}=0\ ,  \braket{\eta,\D_u\sigma}=0\ .
$$

\begin{proposition} Let $k$ be a positive integer.
	Let $\sigma$ be a $C^k$-isotropic surface, then there exists a unique $C^{k-1}$-dual isotropic surface. 
\end{proposition}
\begin{proof}
Let us consider the two-dimensional space 
$$V(s)\defeq\operatorname{Im}\D\sigma(s)\ .
$$
Then $V$ has signature $(1,1)$ and is transverse to its $\bq$-orthogonal $V^\perp$ which also has signature $(1,1)$. The split structure on $V$ defines a split structure on $V^\perp$. We then define uniquely $\eta$ such that $(\sigma,\eta)$ is a split basis of $V^\perp$ and $\braket{\sigma,\eta}=1$. By construction, if $\sigma$ is $C^k$ then $V$ is $C^{k-1}$, hence, $\eta$ is $C^{k-1}$.
\end{proof}

We observe that if $\sigma$ is $C^2$, then $\sigma$ is the dual isotropic surface to its dual isotropic surface.

Following the terminology of Krasnov and Schlenker, we define
\begin{definition}\label{def:KrasnovSchlenker}
Given vector fields $X,Y$ on $S$, we define 
\begin{enumerate}
\item The {\em first fundamental form at infinity} as \[\I^*(X,Y)\defeq \langle\D_X\sigma,\D_Y\sigma\rangle~.\]
\item The {\em second fundamental form at infinity} as \[\II^*(X,Y)\defeq -\langle \D_X\sigma,\D_Y\eta\rangle~.\]
\item The {\em third fundamental form at infinity} as \[\III^*(X,Y)\defeq  \langle \D_X\eta,\D_Y\eta\rangle~.\]
\item The {\em shape operator at infinity} $B^*$ is defined by \[\II^*(X,Y)\defeq \I^*(\B^*(X),Y)\]
 \end{enumerate}
\end{definition} 
We then observe that \[\III^*(X,Y)\defeq \I^*(\B^*(X),\B^*(Y))\ .\]

\section{Isotropic, holonomic and Epstein surfaces}

Let $(W,\q)$ be a real vector space of four dimensions equipped with a quadratic form of signature $(2,2)$ as before.

\subsection{Anti-de Sitter geometry}
The {\em anti-de Sitter} $3$-space is
\[\AdS  = \{ x\in \Proj(W)~\vert~\q(x)<0\}\ ,\]
its double cover is
\[
\AdS_+ = \{ x\in W~\vert~\q(x)=-1\} \ ,
\]
and the covering involution $\iota$ is given by the action of $(-\Id)$ in $\SO(\bq)$.

Given a point $x$ in $\AdS_+$, the tangent space $\T_x\AdS_+$ is identified with the $\q$-orthogonal of $x$, so the restriction of $\q$ defines an $\SO(\bq)$-invariant metric of signature $(2,1)$ and curvature $-1$. Moreover, this metric is $\iota$-invariant and therefore descends to a Lorentz metric on $\AdS$.

Finally, note that the Einstein torus $\Ein $ is the boundary of $\AdS $ in $\Proj(W)$.

\subsubsection{Frame bundle} 
An {\em orthonormal frame} of a 3-dimensional oriented vector space equipped with a signature $(2,1)$ quadratic form $Q$ is an oriented basis $\epsilon=(\epsilon_1,\epsilon_2,\epsilon_3)$ of pairwise $Q$-orthogonal vectors with $Q(\epsilon_1)=Q(\epsilon_2)=-Q(\epsilon_3)=1$. Such an orthonormal frame is fully defined by the first two vectors $(\epsilon_1,\epsilon_2)$. We define in this way the {\em (orthonormal) frame bundle} $\cF(M)$ of an oriented Lorentz 3-manifold as the set of pairs $(x,\epsilon)$ where $x$ is a point of $M$ and $\epsilon=(\epsilon_1,\epsilon_2,\epsilon_3)$ is a frame of $\T_xM$. By projecting any frame $\epsilon$ to $\epsilon_1$, we obtain an $\mathsf{SO}(1,1)$-principal bundle
$$
\cF(M)\mapsto \mathsf U M\ ,
$$
where $\mathsf{U}M \defeq \{ (x,n)\in \T M~\vert ~\q(n)=1\}$ is the {\em (spacelike) unit tangent bundle} of $M$. The Levi-Civita connection of $M$ induces a natural connection on this bundle. Moreover, for $M=\AdS_+$, we observe that $\SO(\bq)$ acts simply transitively on $\cF(W)$ turning the latter into an $\SO(\bq)$ torsor.

\subsection{Unit tangent bundle, holonomic and isotropic surfaces}\label{sec:holon}

\subsubsection{Holonomic surfaces} We briefly recall basic results on jet bundles and contact geometry. We refer to \cite{Saunders:1989aa} for details.

Let $M$ be any manifold of dimension $n$, equipped with a nondegenerate inner product $g$ of type $(p,q)$. Recall that the (spacelike) unit tangent bundle is
$$
\mathsf U M=\{u\in \T M\mid g(u,u)=1\}\ .
$$ 
Later on, we will only consider the case of $M=\AdS_+$, but the discussion is better understood in full generality.

Our goal is to define holonomic surfaces in the spirit of the corresponding definition in Riemannian geometry.

We first recall the construction of the {\em contact form} or {\em Liouville form} of $\mathsf UM$. The construction runs as follows: $\T M$ inherits a symplectic form $\omega$ by duality (defined by the quadratic form) with $\T^*M$. Then the contact form (or {\em Liouville form}) on $\mathsf U M$ is $i_X\omega$ where $X$ is the vector field generating the (spacelike) geodesic flow on $\mathsf U \AdS_+$.  
 The kernel of this Liouville form is called the {\em contact distribution}: it is a field of hyperplanes in the tangent space of $\mathsf UM$. 
 
A submanifold  $\Sigma$ of dimension $n-1$ in $\mathsf{U}M$ is {\em holonomic} if it is always tangent to the contact distribution. The following classical lemma, whose proof is left to the reader and can be found in \cite[Theorem 4.3.15] {Saunders:1989aa}, helps to interpret what a holonomic submanifold is.

\begin{lemma}
Let $S$ be an immersed submanifold in $M$ of type $(p-1,q)$ and $n$ its normal vector field. Then $n(S)$ is a holonomic submanifold in $\mathsf U M$.

Conversely, if $\Sigma$ is a holonomic submanifold in $\mathsf U M$ transverse to the fiber of the projection to $M$, then $\Sigma=n(S)$ where $S$ is an immersed surface in $M$ of type $(p-1,q)$. 
\end{lemma}

We call a holonomic surface of the form $n(S)$, where $S$ is a $(1,1)$-surface in $M$, a {\em typical holonomic surface associated with $S$}.

We will sometimes keep track of the type of the underlying submanifold and speak about $(p-1,q)$-holonomic submanifold.

\subsubsection{Back to anti-de Sitter geometry}
We now specify this discussion for the (spacelike) {\em unit tangent bundle} of $\AdS_+$.

We have the following identification:
\begin{equation}\label{eq:SplittingUnitTangent1}
\mathsf{U}\AdS_+ = \{(x,n)\in W\times W~\vert~\langle x,n\rangle=0~,~\q(x)=-\q(n)=-1\}~.
\end{equation}
Under this identification, the tangent space splits as
\begin{equation}\label{eq:SplittingUnitTangent2}
\T_{(x,n)}\mathsf{U}\AdS_+ = \left\{(u_1,u_2)\in W\times W~\vert~ \langle u_1,x\rangle=\langle u_2,n\rangle = \langle u_1,n\rangle+\langle x,u_2\rangle=0\right\}~.
\end{equation} 

The involution $\iota$ lifts to $\mathsf{U}\AdS_+$, and the quotient is $\mathsf{U}\AdS $, the unit tangent bundle of $\AdS $.
In this identification again, the contact distribution is given by
$$
\ker(\lambda_{(x,n)})=\{(u_1,u_2)\mid \braket{n,u_1}=\braket{x,u_2}=0\} .
$$

\subsubsection{Holonomic and isotropic surfaces}
Finally, in this situation, holonomic surfaces are identified with isotropic surfaces. A direct computation gives the following proposition.

\begin{proposition}\label{pro:holono-isotr}
	If $\sigma$ defines an isotropic surface, let $\eta$ be the dual isotropic surface. Then
	\[
	\function{(x,n):}{S}{\mathsf{U}\AdS_+}{s}{\left(\frac{\sqrt{2}}{2}(\sigma(s)-\eta(s)),\frac{\sqrt{2}}{2}(\sigma(s)+\eta(s)\right))}
	\]
	is a holonomic surface.
	
	Conversely, if $(x,n)$ is a holonomic surface then 
	\[
	\function{(\sigma,\eta):}{S}{W\times W}{s}{\left(\frac{\sqrt{2}}{2}(x(s)+n(s)),\frac{\sqrt{2}}{2}(n(s)-x(s))\right)}
	\]
	is such that $\sigma$ is an isotropic surface and $\eta$ is the dual isotropic surface.
	It follows that any isotropic surface gives a holonomic surface.	
\end{proposition}
The following proposition identifies the first fundamental form at infinity.

\begin{proposition}\label{pro:first-infinity}
	The first fundamental form at infinity of a typical holonomic surface associated with a surface $S$ in $\AdS$ is $\frac{1}{2}(\I +2\II +\III)$ where $\I$, $\II$ and $\III$ are respectively the first, second, and third fundamental forms of $S$.
\end{proposition}
\begin{proof}
	We have by the previous proposition that $\sigma$ is given by $\frac{\sqrt{2}}{2}(x+n)$. Thus
\begin{align}
	\braket{{\rm D}_u\sigma, {\rm D}_u\sigma}
	&=\tfrac{1}{2}\left(\braket{{\rm D}_ux, \D_ux}+2\braket{{\rm D}_ux, {\rm D}_un}+\braket{{\rm D}_un, {\rm D}_un}\right)\\
	&=\tfrac{1}{2}(\I(u,u) +2\II(u,u) +\III(u,u))\ .
\end{align}
This completes the proof.
\end{proof}

\begin{proposition}
	Let $S$ be a holonomic surface immersed in $\U\AdS$ that we consider as a subset of $W\times W$. The first fundamental form at infinity on $S$ is the induced metric from the metric on $W\times W$ given by the quadratic form
	$$
	Q((u,v))=\tfrac{1}{2}\braket{u+v,u+v}\ .	$$
\end{proposition}
\begin{proof}
	Indeed, the first fundamental form on $S$ is given by 
	$$\tfrac{1}{2}\left(\braket{{\rm D}_ux, \D_ux}+2\braket{{\rm D}_ux, {\rm D}_un}+\braket{{\rm D}_un, {\rm D}_un}\right)=\tfrac{1}{2}\braket{{\rm D}_ux + {\rm D}_un,{\rm D}_ux+ {\rm D}_un}\ .$$ The result follows.
\end{proof}

\subsection{Epstein surface}

\begin{definition} Let $\phi$ be an immersion of $S$ in $\Ein$ such that $\phi^*\tau$ is trivial. Let $g$ be a $C^2$ metric on $S$. Then the {\em $g$-Epstein surface} $\epsilon_g$ is the holonomic surface in $\AdS_+$ associated (by Proposition \ref{pro:holono-isotr}) to the isotropic surface $\sigma$ defined by $g$ (by Theorem \ref{theo:isotr-metr}). 	
\end{definition}

We now relate this definition for the sake of completeness to the notion of envelope of horospheres, thus showing that our  Epstein surfaces are the analogs in $\AdS$ of the classical Epstein surfaces constructed in the hyperbolic $3$-space. 

\subsubsection{Horospheres}

Any isotropic vector $x_0$ in $W$ defines a {\em horosphere} in $\AdS_+$ via the formula
\[H(x_0) \defeq \left\{ p\in \AdS_+ ~\vert~ \langle x_0 , p\rangle = -\tfrac{\sqrt 2}{2} \right\}~.\]

A horosphere in $\AdS $ is then the projection of a horosphere in $\AdS_+$. Observe that the vectors $\pm x_0$ define the same horospheres in $\AdS $.

The tangent space to $H(x_0)$ at a point $p$ is identified with $\text{span}\{x_0,p\}^\bot$ and so the induced metric on $H(x_0)$ is Lorentzian. Moreover, $H(x_0)$ is intrinsically flat.

\subsubsection{Envelope}\label{sec:env} Let $S$, $\phi$, $g$ and $\sigma$ be as in the definition.
We then have a family of horospheres
$$
\cH(g)=\{H(\sigma(s))~\vert~s\in S\}.
$$
An {\em envelope} of $\cH(g)$ is then a smooth map $\epsilon_g$ from $S$ to $\AdS $ such that for every $x$

\[\epsilon_g(x) \in H(\sigma(x)) ~\text{ and }~\d_x \epsilon_g(\T_xS) \subset \T_{\epsilon_g(x)}H(\sigma(x))~.\]

\begin{proposition}
	Given a $g$-Epstein surface $\psi$, the map $\pi\circ\psi$ is an envelope for $\cH(g)$, where $\pi$ is the projection to $\AdS$.
\end{proposition}

\begin{proof}
Let $g$, $\sigma$ and $S$ be as in the definition, and let $\eta$ be the dual isotropic surface to $\sigma$ as in paragraph \ref{ss:dual_isotropic}.	
In particular, for any vector fields $X,Y$ on $S$ we have
\[ \langle \sigma,\eta \rangle -1 = \langle \sigma,\sigma \rangle = \langle \eta,\eta\rangle = \langle \D_X\sigma, \eta\rangle = 0 ~ \ \text{ and }~ \langle \D_X\sigma,\D_Y\sigma\rangle=g(X,Y)~.\]
We have
\[\function{
\pi\circ \psi:}{S}{\AdS  \ , }{x}{\frac{\sqrt 2}{2}(\sigma(x)-\eta(x))\ ,}\]
so
\[\langle \sigma(x) , \pi\circ\psi(x)\rangle = -\tfrac{\sqrt{2}}{2}~,\]
and for any tangent vector $u$ in $\T_x S$ 
\[\langle \d_x(\pi\circ \psi) (u) , \sigma(x) \rangle = \langle \d_x(\pi\circ \psi) (u) , \eta(x) \rangle = 0~. \]
So $\pi\circ\psi$ is an envelope for $\cH(g)$.
\end{proof}

One can easily see from the above computation that the induced metric on $\pi\circ \psi$ is equal to $\frac{1}{2} \I^* + \II^* + \frac{1}{2} \III^*$. In particular, it is not always immersed. We denote by $\Sigma_g$ its image.

\section{The $\cW$-volume}

We define in this section the $\cW$-volume for  oriented  3-manifolds with boundary in $\AdS$,  and more generally for the equivariant situation. As in \cite{Bridgeman:2025aa}, we work in a noncompact setting and this requires some technical adjustments. Also the $\cW$-volume is not actually defined for 3-manifolds with boundary in $\AdS$, but rather for 3-manifolds with boundary immersed in $\U\AdS$ with some holonomy condition on the boundary.
\subsection{Preliminary: the geometry of $\U\AdS$ and cobordism} 
\subsubsection{Differential forms on $\mathsf U\AdS$} The unit tangent bundle $\mathsf{U} \AdS $ is equipped with a set of differential forms that we now describe. We use the decomposition described in equation \eqref{eq:SplittingUnitTangent2} and write $u=(u_1,u_2)$ for the two components of a vector $u$ in $\T_{(x,n)}\left(\mathsf{U}\AdS_+\right)$. We introduce and consider the following forms.

\begin{enumerate}
	\item The pull-back $\omega$ of the volume form on $\AdS_+$ via the projection is a closed $3$-form on $\mathsf{U}\AdS_+$ whose value at $(x,n)$ is given by
	\[\omega(u,v,w)= \pi^*\vol_{\AdS}(u,v,w) = \det(x,u_1,v_1,w_1)~.\]

	\item The $2$-form  $\alpha$ whose value at $(x,n)$ is given by
	\begin{align*}
\alpha(u,v) & =	\ \frac{1}{4}\left( \det(x,n,u_2,v_1) +\det(x,n,u_1,v_2) \right)~.	\end{align*}
\end{enumerate}
\vskip 0.2 truecm
We remark that all these differential forms are invariant under the involution $\iota$ and so descend to forms on $\mathsf{U}\AdS $ that we denote the same way.

\subsubsection{Cobordism constant outside a compact}\label{sec:equal-at-infinity}

As a first example of cobordism constant outside a compact,  we have  a  $C^k$ map $\phi$ from $M$ into $\mathsf U \AdS_+$, where
\begin{itemize}
\item $M = S \times [0,1]$ and $S$ is a possibly noncompact   oriented  surface,
\item $\phi(S \times \{0\})$ and $\phi (S \times \{1\})$ are holonomic surfaces, 
\item there exists $K$ a compact subset of $S$, such that $\phi(x,t)$ is constant in $t$ for all $x$ not in $K$. \end{itemize}
More generally, we allow the compact $K$ to have some more complicated topology   and we furthermore consider the equivariant case.

We first recall what we mean by equivariance:
let $N$ be a manifold, $\tilde N$ a Galois cover, $X$ a space equipped with a $\mG$-action, where $\mG$ is a Lie group. 
We recall that an {\em equivariant map} 
from $N$ to $X$ is a couple $(\phi,\rho)$ where $\phi$ 
is a map from $\tilde N$ to $X$, $\rho$ is a group morphism from the Galois group to $\mG$ satisfying 
$$\phi(\gamma\ x)=\rho(\gamma)\phi(x)\ .$$
The choice of the Galois cover is irrelevant: one can always replace the Galois cover by the universal cover. This general definition allows us to restrict any equivariant map on $M$ to a submanifold of $M$.

Now let $N_0$ and $N_1$ be two oriented surfaces, and let $(\phi_0,\rho_0)$ and $(\phi_1,\rho_1)$ be equivariant maps from $N_0$ and $N_1$ respectively to $\mathsf U \AdS_+$ whose images are holonomic surfaces.
A {\em cobordism constant outside a compact}  between  $(N_0,\phi_0,\rho_0)$ and $(N_1,\phi_1,\rho_1)$ is a triple  $(M,\phi,\rho)$, where  
\begin{enumerate}
	\item $M$ is a possibly noncompact oriented 3-manifold with boundary $\partial M= N_0\sqcup \overline N_1$.
	\item $(\phi,\rho)$ is a $C^\infty$ equivariant map from $M$ to ${\mathsf U} \AdS_+$, such that $\phi\vert_{N_i}=\phi_i$ and $\rho\vert_{N_i}=\rho_i$.
	\item Moreover $\phi$ is {\em constant outside a compact}: there is a compact $K$ in $M$ such that $K^c$ is homeomorphic to $U\times[0,1]$, where $U\times\{0\}$, respectively $U\times\{1\}$, is a subset of $N_0$, respectively $N_1$, and $\phi$ is constant in the second factor of $U\times[0,1]$.
\end{enumerate} 

Note that $\phi_1$ and $\phi_0$ agree outside a compact set. Observe that given a cobordism as above $\phi^*\tau$ is a trivial bundle since it is defined on the universal cover of $N$.

\subsection{Differential forms and the $\cW$-volume}
 In this paragraph we choose an equivariant map with values in $X=\mathsf U \AdS_+$ with $\mG$ being the group of isometries of $\AdS_+$. 
\begin{definition}
We define the {\em $\cW$-volume} of the  cobordism $(M,\phi,\rho)$, 
\[\cW(M,\phi) = \int_M \phi^*\omega - \int_{\partial M} \phi^*\alpha~.\]
In this equation, we first observe that $\phi^*\omega$ vanishes outside a compact set. We also use an abuse of language for the last term, since $\phi\vert_{U\times\{0\}}$ and $\phi\vert_{U\times\{1\}}$ agree outside a compact set.  We also observe that $\phi^*\omega$ and $\phi^*\alpha$ being invariant by the Galois group are actually defined on $M$.	
\end{definition}
Because $\omega$ is closed, 
$$
\cW(M,\phi)=\cW(M,\psi)\ ,
$$
whenever $\phi$ and $\psi$ agree outside a compact set in $\operatorname{Int}(M)$. In particular, for a lens cobordism, where we write $\partial M=S_0\sqcup S_1$, $\cW(M,\phi)$ depends only on the restriction of $\phi_0$ on $S_0$, and $\phi_1$ on $S_1$. In that situation, we write
\begin{align}
\cW(S_0,S_1)\defeq\cW(M,\phi)\ \label{eq:defCWSS}.	
\end{align}
We will most of the time write $\cW(M,\phi)\eqdef \cW(M)$. Then note that since one can glue cobordisms, the $\cW$-volume satisfies {\em Chasles relation}
$$
\cW(M_0\sharp M_1)=\cW(M_0)+\cW(M_1)\ .
$$
\subsubsection{Classical formula}

\begin{proposition}\label{pro:class-formula}
Let $(M,\phi,\rho)$ be a cobordism. Assume that the holonomic surface $\phi(\partial M)$ is the set of unit vectors of a  surface $S$ in $\AdS_+$.
 Then
\[\cW(M,\phi) = \vol(\pi\circ\phi(M)) - \frac{1}{2}\int_{S} H\ \d a~, \]
where  $H$ and $\d a$ are, respectively, the mean curvature and the area form of $S$. 
\end{proposition}
\begin{proof} 
If $f$ is an immersion  of type $(1,1)$ from a surface $S$ into $\mathbf{H}^{2,1}_+$, the tangent vectors to the holonomic lift $F$ of $S$ have the form $(u,B(u))$ where $u$ is tangent to $f(S)$ and $B$ is the shape operator. Given a pair of  vectors $(u,v)$ of $S$, we get
\[F^*\alpha(u,v) = \frac{1}{4}(\d a(B(u),v)+\d a(u,B(v))) = \frac{1}{4}\tr(B)\ \d a(u,v)~.\]
The result follows.
\end{proof}

\subsection{Variational formula}

We now prove the variational formula for the $\cW$-volume:

\begin{theorem}[\sc Variational formula]\label{theo:VariationalFormula}
Let $(M,\phi_t,\rho)_{t\in\mathbb R}$ be a smooth family of cobordisms, pairwise equal outside a compact set. Let $g_t$ be the induced first fundamental form at infinity on $\partial M$ and $u$ the compactly supported function such that $\left.\frac{\d}{\d t}\right\vert_{t=0}g_t=2ug_0$. 
Then
	\[\left.\frac{\d}{\d t}\right\vert_{t=0} \cW(M,\phi_t,\rho) =-\frac{1}{2} \int_{\partial M} u F_g~.\]
\end{theorem}

\subsubsection{More differential forms}
We need to introduce more differential forms and objects to understand the variations of $\omega$ and $\alpha$.
\begin{enumerate}
	\item The $1$-forms $x^*$ and $n^*$ given by
	\[x^*(u)=\langle x,u_2\rangle~, \]
	\[n^*(u)=\langle u_1,n\rangle~.\]
	Recall that $x^*+n^*=0$ and that the kernel of any of these forms is the contact distribution  $\cP$ in $\U\AdS$. As a vector subspace of $W\times W$,
	\[\cP_{(x,n)}=\{(u,v)\mid \braket{x,u}=\braket{x,v}=\braket{n,u}=\braket{n,v}=0\}\]
	\item The $2$-forms $\theta_1,\theta_2$ and $\alpha$ whose value at $(x,n)$ is given by
	\begin{align*}
	\theta_1(u,v) & = \det(x,n,u_1,v_1)~,\\
	\theta_2(u,v) & = \det(x,n,u_2,v_2)~.
	\end{align*}
	\item  Let $\cF(\cQ)$  be the space of pairwise orthogonal triples $(x,u,n)$  satisfying
	\[-\braket{x,x}=\braket{n,n}=\braket{u,u}=1\ .\]  
Let $\beta$ be the  1-form  on $\cF(\cQ)$  defined by 
	\begin{equation}\label{eq:beta}
\beta_{(x,n,u)}(w)=-\det(x,n,u,w_3)\ , 
\end{equation}	\end{enumerate}
\vskip 0.2 truecm
\noindent{\sc Remarks}
\begin{enumerate}
	\item All these differential forms are again invariant under the involution $\iota$ and therefore descend to forms on $\mathsf{U}\AdS $ that we denote the same way.
	\item We have a projection $\pi$ from $\cF(\cQ)$ to $\U\AdS$ defined by $\pi(x,n,u)=\pi(x,n)$.
\end{enumerate}
\bigskip
We first need to identify precisely $\beta$. Let $\cQ$ be the subdistribution of rank 2 of
	$\cP$ defined by 
	\[\mathcal  Q\defeq\{(u,v)\in \cP\mid u=v\}\ .\] We observe that we have an identification of $\cQ_{x,n}$ with the orthogonal of $(x,n)$ in $W$. Then $\cF(\cQ)$ is the positive unit tangent bundle of $\cQ$. 

\begin{proposition} 
 The form  $\d\beta$ is the curvature form of the bundle $\cQ$, equipped with the  $(1,1)$-metric and the connection from its embedding as a subbundle of the trivial bundle $W\oplus W$ over $\U\AdS$.
 \end{proposition}
 
\begin{proof} By definition, if $(x^t,n^t,u^t_1)$ is a curve in $\cF(\cQ)$ seen as a subset of $W^3$
$$
\beta_{(x^0,n^0,u^0)}(\dot x,\dot n,\dot u)=\braket{\dot u ,e}\ ,
$$
where $e$ is the vector satisfying $\braket{e,e}=-1$, orthogonal to $(x,n,u)$ and such that $\det(x,n,e,u)=1$. Thus 
$$
\beta_{(x^0,n^0,u^0_1)}(\dot x,\dot n,\dot u)=-\det(x,n,u,\dot u) .
$$
This is what we wanted to prove.
\end{proof}

\begin{proposition}[\sc Fundamental equations]\label{pro:FundamentalEquations}
	We have 
	\begin{align}
		 \d\beta & =\theta_1-\theta_2\ ,\\
	 \d\alpha & = \frac{1}{2}(n^*\wedge\theta_2-x^*\wedge \theta_1)\ ,
\end{align}
\end{proposition}

\begin{proof} We first compute $\d\beta$. Recall that
	$$
	\d\beta(w,v)=\D_w\beta(v)-\D_v\beta(w).
	$$ 
From equation \eqref{eq:beta}, we have 
	\begin{align*}
		\D_w\beta(v)&\defeq - \det(w_1,n,u,v_3)- \det(x,w_2,u,v_3) - \det(x,n,w_3,v_3)\ .
	\end{align*} 
	Recall that $\cF(\cQ)$ is the subset of $W^3$ consisting of triples $(x,n,u)$ which are pairwise orthogonal and satisfy
	$$
	-\braket{x,x}=\braket{n,n}=\braket{u,u}=1\ .
	$$
Differentiating these equations, and introducing the vector $e$ orthogonal to $x$, $u$ and $n$, such that $\det(x,n,u,e)=1$, recall that $\braket{n,n}=\braket{u,u}=1$ while   $\braket{x,x}=\braket{e,e}=-1$
 \begin{itemize}
 \item $\braket{u,w_3}=0=\braket{v_3,u}$ and thus \begin{align}\det(x,n,w_3,v_3)=0\ \label{eq:dbeta2}.\end{align}
 \item $\braket{w_2,n}=0$ and thus  \begin{align}
\det(x,w_2,u,v_3)&=\braket{v_3,n}\det(x,w_2,u,n)\cr&=-\braket{w_2,e}\braket{v_3,n}\det(x,e,u,n)\cr&=\langle w_2,e\rangle\langle v_3,n\rangle\ \label{eq:dbeta3}.\end{align}
 \item   $\braket{w_1,x}=0$ and thus  \begin{align}
\det(w_1,n,u,v_3)&=-\braket{v_3,x}\det(w_1,n,u,x)\cr&=\braket{w_1,e}\braket{v_3,x}\det(e,n,u,x)\cr&=-\langle w_1,e\rangle\langle v_3,x\rangle\ \label{eq:dbeta4}.\end{align} \end{itemize}
It follows that 
 \begin{align*}
		\D_w\beta(v)		&=\braket{w_2,e}\braket{v_3,n}-\braket{w_1,e}\braket{v_3,x}\ .
\end{align*} 
 Finally, differentiating $\braket{u,x}=\braket{u,n}=0$, we get
 $$
 \braket{v_3,x}+\braket{u,v_1}=0\ , \  \braket{v_3,n}+\braket{u,v_2}=0\ ,
 $$
hence
\begin{align}
\D_w\beta(v) & =\braket{w_1,e}\braket{v_1,u}-\braket{w_2,e}\braket{v_2,u} \\ \D_v\beta(w) &=\braket{v_1,e}\braket{w_1,u}
-\braket{v_2,e}\braket{w_2,u}.
\label{eq:dbet5}	
\end{align}
Since
\begin{align}
\det(x,n,w_2,v_2)&=\langle w_2,e\rangle\langle v_2,u\rangle-\langle w_2,u\rangle\langle v_2,e\rangle~,\\	
\det(x,n,w_1,v_1)&=\langle w_1,e\rangle\langle v_1,u\rangle-\langle w_1,u\rangle\langle v_1,e\rangle~,	
\end{align}
it follows that
\begin{align*}
\d\beta(w,v)&=\D_w\beta(v)-\D_v\beta(w) = \det(x,n,w_1,v_1)-\det(x,n,w_2,v_2)\ \\	
&=\theta_1(w,v)-\theta_2(w,v)\ .
\end{align*}
Let us now compute the second identity. Recall that 
$$
\alpha(u,v)=\frac{1}{4}\left( \det(x,n,u_2,v_1) +\det(x,n,u_1,v_2) \right)~.
$$
We have 
$$
\d\alpha(u,v,w)=(\D_u \alpha)(v,w) +(\D_v\alpha)(w,u)+(\D_w\alpha)(u,v)\ .
$$
Then
\begin{align*}
	4(\D_u \alpha)(v,w)&=\det(u_1,n,v_2,w_1)+\det(x,u_2,v_2,w_1)\\&+\det(u_1,n,w_2,v_1)+\det(x,u_2,w_2,v_1)\ .
\end{align*}

Now we use the fact that $u_1$, $v_1$ and $w_1$ are all normal to $x$, while $u_2$, $v_2$ and $w_2$ are  normal to $n$.
Thus 
\begin{align*}
4(\D_u \alpha)(v,w)&=-\braket{v_2,x}\det(u_1,n,x,w_1)+\braket{w_1,n}\det(x,u_2,v_2,n)\\
&- \braket{w_2,x}\det(u_1,n,x,v_1)+\braket{v_1,n}\det(x,u_2,w_2,n)\\
		&=-\braket{v_2,x}\theta_1(w,u)+\braket{w_1,n}\theta_2(u,v)-\braket{w_2,x}\theta_1(u,v)+\braket{v_1,n}\theta_2(u,w)\ .
	\end{align*}
	Thus 	$4\d\alpha=2n^*\wedge\theta_2-2x^*\wedge \theta_1$. 
\end{proof}

The following  proposition identifies $\beta$ as the connection form of a holonomic surface.
Let $S$ be a holonomic surface. We have a linear map from $\T S_{(x,n)}S$ to $\cQ_{(x,n)}$, both seen as subspaces of $W\oplus W$, given by 
	\[\Lambda:(u,v)\mapsto \tfrac{1}{2}(u+v,u+v)\ ,\]
	\begin{proposition}\label{pro:holoI}
		The linear map $\Lambda$ is an isometry from $\T_{x,n} S$ equipped with its first fundamental form at infinity to $\cQ_{(x,n)}$ equipped with the induced metric from $W\oplus W$.
	\end{proposition}
	\begin{proof}
		Let $U=({\rm D}_ux,{\rm D}_u n)$ be a tangent vector to $S$. Then
		\[
		\braket{\Lambda(U),\Lambda(U)}=\tfrac{1}{2}\braket{{\rm D}_ux+{\rm D}_u n,{\rm D}_ux+{\rm D}_u n}=\I^*(U,U)\ .\]
		This completes the proof.
	\end{proof}

\subsubsection{Variation of surfaces}

To show Theorem~\ref{theo:VariationalFormula}, we will use the following lemma. 
Let $\sigma_t$ be the isotropic surface associated with $g_t\eqdef e^{2u_t}g$ and $\eta_t$ its dual surface. Write $(\sigma,\eta)=(\sigma_0,\eta_0)$ and denote by $\dot\sigma$ and $\dot\eta$ the derivative at $t=0$.

\begin{lemma}\label{lem:var-se}
Using the above notations, we have
\[\dot\sigma= u\sigma~\text{ and }~\dot\eta = -u\eta - \nabla^gu~,\]
where $\nabla^gu$ is the $g$-gradient of $u$, that is the vector field on $S$ satisfying $g(\nabla^gu,\cdotp)=\d u$.
\end{lemma}

\begin{proof}
By construction, the sections $(\sigma_t,\eta_t)$ satisfy
\[\sigma_t  = e^{u_t}\sigma ~\ ,~\ 
\langle \eta_t,\eta_t\rangle  = 0~\ ,~\text{ and }~
\langle\sigma_t,\eta_t\rangle  = 1~.\]
Differentiating the first equation gives $\dot\sigma=u\sigma$. The second gives $\braket{\dot\eta,\eta}=0$ and the third gives $\langle \dot\eta ,\sigma\rangle = -u~$.
In particular, there is a vector field $X$ on $S$ such that
\[\dot\eta = -u\eta + \D_X \sigma~.\]

We now obtain the expression of $X$. Let $Y$ be a tangent vector on $S$. Then the last equation to be used is 
\[\langle \D_Y\sigma_t,\eta_t\rangle = 0\ .\]
We differentiate in time the last equation to get 
\[\braket{\D_Y\dot\sigma,\eta}+ \braket{\D_Y\sigma,\dot\eta} = 0\ .\]
Using the previous identifications, we obtain
\[0=\d u(Y)+u\braket{\D_Y\sigma,\eta}-u\braket{\D_Y\sigma,\sigma}+\braket{\D_Y\sigma,\D_X\sigma}=\d u(Y)+\braket{\D_Y\sigma,\D_X\sigma}\ .\]
Using the equality $\langle \D_X\sigma,\D_Y\sigma\rangle = g(X,Y)$, we finally obtain 
\[\d u(Y)+g(X,Y)=0\ ,\]
which yields $X=-\nabla^g u$ and the result.
\end{proof}

\begin{proof}[Proof of Theorem~\ref{theo:VariationalFormula}]
 Let $(\Sigma_t)_{t\in\mathbb R}$ be the family of holonomic  surfaces associated with $e^{2u_t}g$. From Proposition~\ref{pro:holono-isotr} and Lemma \ref{lem:var-se}, we see that the  variation $\zeta$ of the family of holonomic  surfaces along $\Sigma_0$ is given by
\begin{align*}
 \zeta&=(\zeta_1,\zeta_2)\\&=\frac{\sqrt 2}{2}\left(u \sigma+u\eta+\nabla^gu,u\sigma-u\eta-\nabla^gu\right)\\&=\left(nu+\frac{\sqrt{2}}{2}\nabla^gu,xu-\frac{\sqrt{2}}{2}\nabla^gu\right)~.	
\end{align*}
Given a family of compact cobordisms $\phi_t$,  we obtain
$$
\left.\frac{\d}{\d t}\right\vert_{t=0} \cW(M,\phi_t,\rho) = - \int_{S}\phi^*\iota_\zeta(\omega-\d\alpha)\ .
$$
A first computation gives for $U$ and $V$ tangent to $\Sigma_0$
\begin{align*}
	\iota_\zeta\omega(U,V)&=\det(x,\zeta_1,U_1,V_1)\\&=u\det(x,n, U_1,V_1)
	\\
	&=u\theta_1(U,V)\ ,
\end{align*}
where we used the fact that $\nabla^gu$ is tangent to $\Sigma_0$.
Similarly for $U$ and $V$ tangent to $\Sigma_0$
\begin{align*}
2\iota_\zeta\d\alpha(U,V)&=(n^*\wedge\theta_2-x^*\wedge\theta_1)(\zeta,U,V)\\
&=\braket{n,\zeta_1}\ \theta_2(U,V)-\braket{x,\zeta_2}\ \theta_1(U,V)\\
&=u\left(\theta_2(U,V)+\theta_1(U,V)\right)\ .
 	\end{align*}
 	It then follows that 
 	$$
 	\iota_\zeta(\omega-\d \alpha)=\frac{1}{2}u(\theta_1-\theta_2)=\frac{1}{2}u\ \d\beta\ ,
 	$$
 	where the last equation comes from Proposition~\ref{pro:FundamentalEquations}.
 	Now $\d\beta$ is the curvature of the bundle $\cQ$ over 
 $S$, thus from Proposition~\ref{pro:holoI}, $\d\beta$ is also the curvature of the bundle $\T S$ equipped with the first fundamental form at infinity. 
\end{proof}

\section{Liouville action} 

We define in this section the {\em Liouville action} for Lorentzian metrics on a surface equal outside a compact set, relate it to $\cW$-volumes for {\em product cobordisms} of surfaces that can be equivariantly immersed in $\Ein$, and then concentrate on compact surfaces. In the next section, we will focus on the split annulus, where we can relax the condition "equal outside a compact set" and introduce the {\em $\cS$-class}.
\subsection{Liouville action for metrics equal outside of a compact set} Let $S$ be a surface equipped with a split structure $\boldsymbol \sigma$. We define $\cM(S,\sigma)$ to be the set of metrics conformal to $\sigma$. When two metrics $g$ and $h$ belong to $\cM(S,\sigma)$ we say that they are {\em equal outside a compact set} if we can write $h=e^{2u} g$ with $u$ of compact support.

We then define 
\begin{definition}
	Let $g$ and $h$ be two metrics in $\cM(S,\sigma)$ equal outside a compact set with $h=e^{2u}g$. The Liouville action of $(g,h)$ is then defined as
	\[\cS(g,h)= \frac{1}{2}\int_S u F_g - \frac{1}{4}\int_S u\ \d (\d u\circ \I)~.\]

	\end{definition}
\begin{remark} \label{rem:liouville_physics}
     If $g = \d x \ \d y$ and $u$ are compactly supported, then 
    $$\cS (g, e^{2u} g) = -\frac{1}{2} \int_S  u \ \partial^2_{xy} u \ \d x \ \d y =  \frac{1}{2} \int_S  \partial_x u \ \partial_{y} u \ \d x \ \d y.$$
     More generally, if $u$ is compactly supported, the expression in Proposition~\ref{prop:liouville_expression} below can be rewritten as
    $$\cS (g, e^{2u} g) = \frac{1}{4} \int_S  (2 u K_g - u \dal_g u ) \ \d a_g,$$
    where $\d a_g$ is the area form of $g$.  
     This expression coincides with the Liouville action with zero cosmological constant in the physics literature; see, e.g., \cite{Teschner:2001ds}.
\end{remark}	

	We now summarize its main properties.

\begin{proposition} \label{prop:liouville_expression}
Let $g$ and $h$ be two metrics  in $\cM(S)$ with $h=e^{2u}g$ for a compactly supported smooth function $u$, then
\begin{enumerate}
\item we have the {\em variational formula}: for a family of metrics $h_t=e^{2u_t}g$
	$$
	\left.\frac{\rm d}{\rm d t}\right\vert_{t=0}\cS(g,h_t)= \frac{1}{2}\int_S\dot u\  F_{g}\ ,
	$$ 
	where $\dot u$ is the function defined by $\dot u  =\left.\frac{\rm d}{\rm d t}\right\vert_{t=0}u_t$.
	\item we have 
\begin{align}
	&\hbox{\sc [Chasles formula]}\ \ \ \ \ \cS(h_1,h_2)+\cS(h_2,h_3)=\cS(h_1,h_3)\ ,\\
	&\hbox{\sc [Monotonicity formula]}\ \ \ \ \cS(g,h)=\frac{1}{4}\int_S u(F_g+F_h)\ .
	\end{align}
	\item For any split diffeomorphism $\varphi$ of $S$, we have the {\em split invariance}
	$$\cS(\varphi^*g,\varphi^*h)=\cS(g,h)~.$$
	\end{enumerate}
\end{proposition}

\begin{proof} For the first item, by definition
	\[\cS(g,h_t)= \frac{1}{2}\int_S u_t F_g - \frac{1}{4}\int_S u_t\ \d (\d u_t\circ \I)~,\] since $u_0=0$ the variational formula immediately holds. 
	
For Chasles formula, recall that by	proposition \ref{pro:ConformalChange}, we have if $h=e^{2u}g$ then  $\d(\d u \circ\I)=F_g-F_h$. Thus if $h=e^{2u}g$ 
 and $k=e^{2v}h$
\begin{align*}
\cS(g,h)+\cS(h,k)&=\frac{1}{4}\int_S 2(uF_g+vF_h)- u\d(\d u\circ \I)-v	\d(\d v\circ \I)\\
&=\frac{1}{4}\int_S 2((u+v)F_g) -2v \d(\d u \circ \I)- u\d(\d u\circ \I)-v	\d(\d v\circ \I)\\
&=\frac{1}{4}\int_S 2((u+v)F_g) -(u+v) \d(\d (u+v) \circ \I)\\
&=\cS(g,k)\ ,
\end{align*}
where we used that $(u,v)\mapsto \int_S u\d(\d v\circ \I)$ is symmetric in the third equality and that $k=e^{2(u+v)}g$ in the last. This proves Chasles formula.

The monotonicity formula is an immediate consequence of the fact that if $h=e^{2u}g$ then $g=e^{-2u}h$ thus using Chasles formula,
$$
\cS(g,h)=\frac{1}{2}(\cS(g,h)-\cS(h,g))=\frac{1}{4}\int_S u(F_g+F_h)\ .
$$	
Alternatively one can use proposition \ref{pro:ConformalChange} directly.
	
 Finally, the split invariance is a direct consequence of the change of coordinates formula. 

\end{proof}

The following corollary of the variational formula justifies the name of Liouville action. 

\begin{corollary}[\sc Critical point]\label{cor:critical_pt} 
   With the same notation as in Theorem~\ref{theo:VariationalFormula},  if $g$ in $ 
    \cM (S)$ is a critical point for all compact Weyl scaling preserving the area, then $g$ has constant curvature. More precisely, 
    if for all $u$ compactly supported functions on $\bA$ such that 
    $\int_S u \d a_g = 0$, we have 
    $$\left. \frac{\d}{\d t} \right |_{t = 0} \cS (g, e^{2tu} g) = 0,$$
    then, $g$ has constant curvature.
\end{corollary}

\subsection{Liouville action and the $\cW$-volume}

In this paragraph we show how the Liouville action between two metrics is related to the $\cW$-volume in many cases.

More precisely, we say that $S$ is {\em equivariantly immersed in $\Ein$} if there is an equivariant immersion from $S$ to $\Ein$.  
\begin{definition}{\sc (Product cobordism)}
We then say a cobordism $\bM=(M,\phi,\rho)$ equal outside a compact set is a {\em product cobordism} if \begin{itemize}
 \item 	$M=S\times [0,1]$. 
 \item The map $(x,0)$ to $(x,1)$ from $\phi_0$ to $\phi_1$ is conformal.
 \end{itemize}
In the situation of a product cobordism $\bM$, we thus obtain two conformal metrics $g^\bM_0$ and $g^\bM_1$ on $S$ as well as a conformal equivariant $\psi$ from $S$ to $\Ein$.	
\end{definition}

Conversely, thanks to Theorem \ref{theo:isotr-metr}, given an equivariant immersion $\psi$ of $S$ in $\Ein$ and two metrics $g_0$ and $g_1$ on $S$ equal  outside of a compact set, we obtain two (equivariantly) immersed holonomic surfaces $\phi_0$ and $\phi_1$ and a product cobordism $\bM_{\psi,g_0,g_1}$ equal outside a compact set between them. 

Thus as a consequence of Theorem \ref{theo:isotr-metr} there is a bijection between product cobordisms on $S\times[0,1]$ and triples $(\psi,g_0,g_1)$ where $\psi$ is an equivariant immersion from $S$ to $\Ein$ and $g_0$ and $g_1$ are conformal Lorentzian metrics on $S$. We thus denote a product cobordism as 
$$
\bM=\bM_{\psi,g_0,g_1}
$$

We can then relate the $\cW$-volume to the Liouville action using both constructions. 

\begin{theorem}\label{theo:cScW}
Given a product cobordism $\bM=(M,\phi,\rho)$ associated with $(\psi,g_0,g_1)$, then 
	$$\cW(M,\phi,\rho)=\cS(g_0,g_1)\ .$$
\end{theorem}
Observe that as a corollary $\cW(M,\phi,\rho)$ does not depend on the equivariant conformal immersion $\psi$.
\begin{proof}
	This result follows from two observations. First, given a product cobordism, $\phi_0=\phi_1$ exactly whenever $g_0=g_1$ due to Theorem \ref{theo:isotr-metr}. Therefore, if $\phi_0=\phi_1$ for a product cobordism
	$$
	\cS(g_0,g_1)=0\ ,
	$$
	and conversely if $g_0=g_1$ and for any immersion $\psi$, for the associated product cobordism we have 
	$$
	\cW(\bM_{\psi,g_0,g_1})=0\ .
	$$
	Now take $g_t=e^{2u_t}g_0$ and $(M,\phi_t,\rho)$ a product cobordism equal to $\bM_{\psi,g_t,h}$.
	The second observation is that using the variational formula  for the Liouville action in Proposition ~\ref{prop:liouville_expression}, and the variation formula for the $\cW$-volume in Theorem \ref{theo:VariationalFormula}, we have
	$$\left.\frac{\d}{\d t} \right\vert_{t=s} \cW(M,\phi_t,\rho)=\frac{1}{2}\int_S \left.\frac{\d}{\d t}\right\vert_{t=s} u_t\ \  F_{g_s}\ =\left.\frac{\d}{\d t}\right\vert_{t=s} \cS(g_s,g_t)=\left.\frac{\d}{\d t}\right\vert_{t=s} \cS(g_0,g_t)\ . $$
		where $g_t=e^{2u_t}g_0$ and $(M,\phi_t,\rho)$ is a product cobordism equal to $\bM_{\psi,g_t,h}$, and using Chasles relation in the last equality. This last observation concludes the proof of the theorem. Indeed, we can find a path of conformal metrics $g_t=e^{2u_t}g_0$ from $g_0$ to $g_1$
	\begin{align}
		\cS(g_0,g_1) &=\int_{0}^1\left.\frac{\d}{\d t}\right\vert_{t=s} \cS(g_0,g_t)\ \d s
		\\&=\int_{0}^1\left.\frac{\d}{\d t} \right\vert_{t=s} \cW(\bM_{\psi,g_0,g_s})\ \d s=\cW(\bM_{\psi,g_0,g_1}).
	\end{align}
	This concludes the proof.
\end{proof}

\subsection{Liouville action for the torus}

We concentrate here on the case $S$ is closed.  For reasons of Euler characteristic, the only case is that the surface $S$ is the torus $T^2$. This will turn out to be an easy case of the theory. In that context, we do not have to worry about compact sets. 

Again, by Gau\ss-Bonnet equation (for the case of Lorentzian metric), the only  metrics of constant curvature  carried by 
 $T^2$ are flat. In the case of the split structure $T^2=S^1\times S^1$, the conformal group of $T^2$ is $\operatorname{Diff}(S^1)\times \operatorname{Diff}(S^1)$ and thus $T^2$ carries an infinite-dimensional space of flat metrics conformal to each other: this is in sharp contrast with the Riemannian case. However, as an immediate consequence of the monotonicity formula in Proposition \ref{prop:liouville_expression}, we have
 \begin{proposition}\label{pro:onTorus}
 	Given two conformal flat metrics $g$ and $h$ on $T^2$, then
 	$$
 	\cS(g,h) = 0\ .
 	$$ 
 \end{proposition}
 We can therefore define the {\em Liouville action for a conformal metric $h$ on $T^2$} by 
 $$
 \cS(h)\defeq \cS(g_0,h)\ ,
 $$
 where $g_0$ is any conformal flat metric on $T^2$. Observe that $\cS(g_0,h)$ does not depend on $g_0$ by Proposition \ref{pro:onTorus} and Chasles Formula in Proposition \ref{prop:liouville_expression}.

\section{Liouville action for the split annulus}

We now concentrate on the case of the split annulus. There are two reasons for this.

\begin{enumerate}
	\item In the case of the split annulus, we can move beyond the class of metrics equal outside a compact set, to a larger equivalence whose equivalence class is called the $\cS$-class.
	\item Moreover, in the case of the split annulus, we have an interesting class of metrics: those which have constant curvature. These metrics are not equal outside a compact set but are in the same $\cS$-class.
\end{enumerate}
Nevertheless, the case of the split annulus is more challenging than the torus case.
In this section, $\bA$ denotes  the split annulus.

\subsection{The $\cS$-class}

We denote by $\cM(\bA)$ the space of {\em weak metrics}, that is, the space of volume forms on $\bA$ defining measures absolutely continuous with respect to the de Sitter volume form.

We also denote by $\cM_b(\bA)$ the set of $C^2$ metrics in $\cM(\bA)$ whose sectional curvature is bounded (namely, in $\mathsf L^\infty(\bA)$).

\begin{definition}\label{def:Sclass}
Let $g$ be a  metric in $\cM_b(\bA)$. The {\em $\cS$-class of $g$} is the set $\cS_g$ of weak metrics $h$ in $\cM(\bA)$ of the form $e^{2u}g$ where $u$ satisfies
\begin{enumerate}
	\item the function $u$ belongs
	 to $\mathsf{L}^\infty(\bA,\d a_g)$ where $\d a_g$ is the area form of $g$ associated with the volume form $\omega_g$,
	\item the function $u$ tends to $0$ uniformly on $\partial \bA$ (that is, for any $\epsilon>0$, we have $\vert u\vert <\epsilon$ away from a compact set),
	\item the d'Alembertian $\dal_g u$ belongs to  $\mathsf{L}^\infty(\bA,\d a_g)$ and $\mathsf{L}^1(\bA,\d a_g)$
	\item  there is a polygonal curve $P$  such that  $\VB(u,P)$ is finite.	\end{enumerate}

Observe that if $u$ is $C^0$, then item (1) follows from item (2).

Moreover, condition (4) is easily satisfied, in particular when $u$ is in $C^1$.

We denote by $\cS_0$ the $\cS$-class of the de Sitter metric $g_0$.
\end{definition}

A \emph{polygonal curve} in $\bA$ is an oriented loop $ P$ made of finitely many lightlike segments and whose vertices project on each factor $\Rp$ to a cyclically oriented tuple (see Figure \ref{fig:polygonal}). We say that a tuple $(\alpha_1,\ldots,\alpha_{2k})$ \emph{represents} $P$ if
\begin{itemize}
\item the projection of $(\alpha_1,\ldots,\alpha_{2k})$ on each factor is cyclically oriented,
 \item for any $i$ the points $\alpha_{2i}$ and $\alpha_{2i+1}$ are the extremities of a vertical segment $[\alpha_{2i},\alpha_{2i+1}]$ contained in $ P$,
 \item for any $i$, the points $\alpha_{2i+1}$ and $\alpha_{2i+2}$ are the extremities of a horizontal segment $[\alpha_{2i+1},\alpha_{2i+2}]$ contained in $ P$,
 \item $ P= \bigcup_{i=1}^{2k}[\alpha_i,\alpha_{i+1}]$ with $\alpha_{2k+1}=\alpha_1$.
\end{itemize}
Observe that we do not assume that $\alpha_i\neq \alpha_{i+1}$, so we have many sequences that represent the same curve. For any function $u$ on $\bA$ and polygonal curve $ P$, we define
\[ \VB(u, P)=\sup \left\{\sum_{i=1}^{n}\vert u(\alpha_{2i-1})-u(\alpha_{2i})\vert \  ,\  (\alpha_1,\ldots\alpha_{2n})\hbox{ representing }  P\right\}\ .\]

\begin{figure}[ht] 
\begin{center}
\includegraphics[width=8cm]{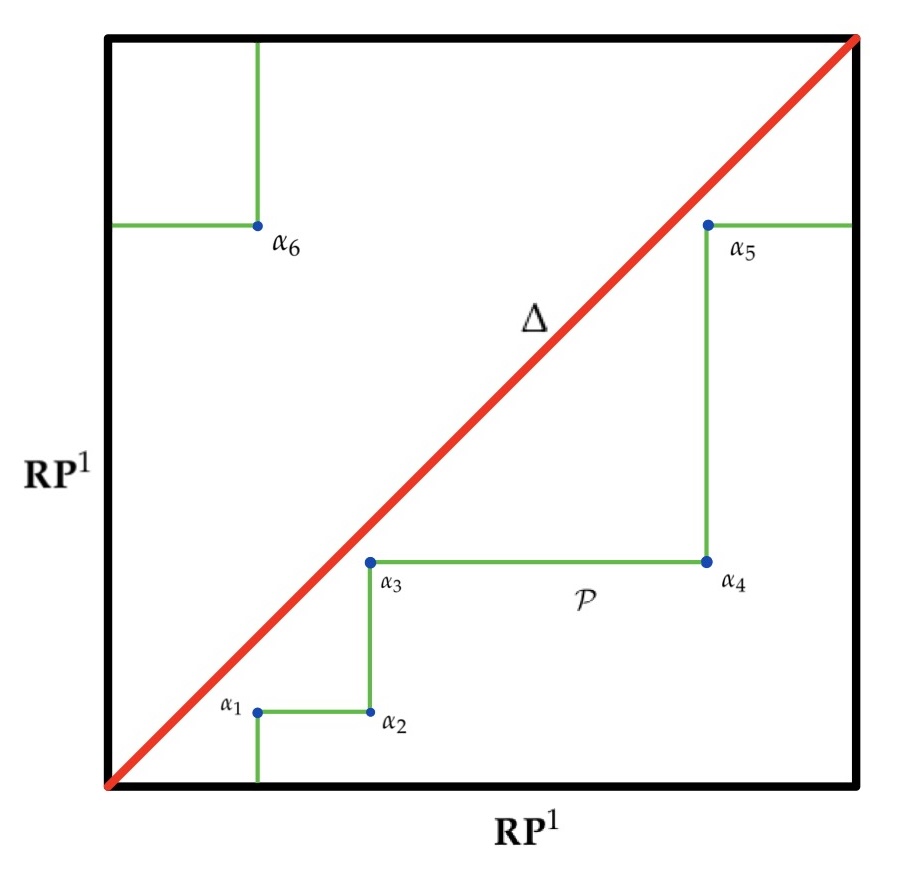}
\caption{A polygonal curve $\mathcal P$ represented by $(\alpha_1,...,\alpha_6)$} 
\label{fig:polygonal}
\end{center}
\end{figure}

\begin{lemma}\label{lem:Sclass-equi}
Let $g$ be a metric in $\cM_b(\bA)$. If $h$ is a $C^2$ metric in $\cS_g$, then $h$ is in $\cM_b(\bA)$ and $\cS_g=\cS_h$. In particular, being $C^2$ and in the same $\cS$-class defines an equivalence relation on $\cM_b(\bA)$.
\end{lemma}

\begin{proof}
Write $h=e^{2u}g$. Using Proposition \ref{pro:ConformalChange} we have
\[\dal_g u = K(g)-e^{2u}K(h)~,\]
so $K(h)$ belongs to $\mathsf L^\infty(\bA)$ since $K(g)$ and $u$ do. It follows that $h$ is in $ \cM_b(\bA)$.

Since $u$ belongs to $\mathsf L^\infty(\bA)$, there exists a positive $A$ such that $A^{-1}<e^{2u}<A$. In particular, since $\d a_h= e^{2u} \d a_g$, we have $\mathsf L^1(\bA,\d a_g)=\mathsf L^1(\bA,\d a_h)$. 

Let $f$ be a metric in $\cS_g$, of the form $f=e^{2v}g$. Then $f=e^{2(v-u)}h$. It follows that $w=v-u$ satisfies items (1) and (2) of the definition above. Moreover, by equation~\eqref{eq:daldal}
$$
\dal_h w=e^{-2u}\dal_g w\ ,
$$
it follows that item (3) of the definition is satisfied. Hence $f$ belongs to $\cS_h$. Similarly for item (4), we observe that the condition is satisfied by $w$ which is $C^2$.

For item (4), let $P$ be a polygonal curve such that $\VB(v,P)<\infty$.
Since $u$ is $C^2$, we have $\VB(u,P)<\infty$. Hence
\[
\VB(w,P)\leq \VB(v,P)+\VB(u,P)<\infty .
\]
The result follows.
\end{proof}

For a weak metric $h$ in the $\cS$-class of $g$ with $h=e^{2u}g$, we write 
$$
F_h\defeq F_g- \d(\d u\circ I)\ .
$$
Observe that this notation is coherent for $C^2$-metrics by Proposition \ref{pro:ConformalChange}; moreover, $F_h$ does not depend on the choice of the $C^2$-metric $g$ such that $h$ belongs to the $\cS$-class of $g$. Finally, one observes that $F_h$ belongs to $L^\infty(\bA)$.
\subsection{The Liouville action}

\begin{definition}\label{def:Liouville}
Let $g$ and $h$ be two metrics in $\cM(\bA)$ of the same $\cS$-class. Write as before $h = e^{2u} g$. The {\em Liouville action} is then
\[\cS(g,h) = \frac{1}{2}\int_\bA uF_g - \frac{1}{4} \int_\bA u\ \d(\d u\circ \I)~.\]
\end{definition}

Observe that since $F_g= K(g) \d a_g$ and $\d(\d u\circ \I) = \dal_g u\ \d a_g$, our assumptions on the $\cS$-class imply that both integrals are defined.

\begin{proposition}\label{pro:propSclass}
Let $g$, $h$, and $k$ be metrics in $\cM_b(\bA)$ in the same $\cS$-class, write $h=e^{2u}g$, and let $\varphi$ be a $C^2$ split diffeomorphism of $\bA$. Then
\begin{enumerate}
	\item The monotonicity formula holds
	\[\cS(g,h) = \frac{1}{4}\int_\bA u(F_g+F_h)~.\]
	\item The metrics $\varphi^* g$ and $\varphi^*h$ are in the same $\cS$-class and 
	\[\cS(\varphi^* g,\varphi^* h)=\cS(g,h)~.\]
	\item Chasles relation holds
	\[\cS(g,h)+\cS(h,k) = \cS(g,k)~.\]
\end{enumerate}
\end{proposition}

\begin{lemma}\label{lem:BV} Given any two embedded curves $P_0$ and $P_1$ then 
\[\left\vert \VB(f,P_0)-\VB(f,P_1)\right\vert\leq \tfrac{1}{2}\int_{\bA}\vert\dal_g(f)\vert\ \d a_g\ .\]
\end{lemma}

\begin{proof} Let   $(\alpha_i)$ and $(\beta_j)$ be sequences representing $P_0$ and $P_1$ respectively, such that $\alpha_i$ and $\beta_i$ are on the same vertical segment. In other words, $(\alpha_{2i-1},\alpha_{2i}, \beta_{2i-1},\beta_{2i})$ are the vertices of a diamond $\Delta$. Observe now that
\[\int_{\Delta}\vert\dal_g(f)\vert\ \d a_g\geq \left\vert\int_{\Delta}\d(\d f\circ I)\ \right\vert=2\big\vert f(\beta_{2i-1})-	 f(\beta_{2i}) +f(\alpha_{2i})-	 f(\alpha_{2i-1})\big\vert\ .
\]Thus 
\[
\vert f(\beta_{2i-1})-	 f(\beta_{2i})\vert \leq  \vert f(\alpha_{2i-1})-	 f(\alpha_{2i})\vert+ \tfrac{1}{2}\int_{\Delta}\vert\dal_g(f)\vert\ \d a_g\ .
\]
The result follows. \end{proof}
This lemma has a useful corollary.
\begin{corollary}\label{cor:BV}
	Let $u$ be a $C^2$ function such that $e^{2u}g$ is  in the $\cS$-class of $g$. Then there is a constant $K_u$ such that for any polygonal curve $P$.
	$$
\VB(u,P)\leq K_u\ .
	$$
\end{corollary}
\begin{proof}
Since $u$ is $C^2$, then for any polygonal curve $P_0$, 
$\VB(u,P_0)$ is finite.	Moreover, by Lemma \ref{lem:BV} we have that for any polygonal curve $P$,
$$
\VB(u,P)\leq \VB(u,P_0)+\tfrac{1}{2}\Vert\dal_g(u)\Vert_1\ .
$$
The result follows.
\end{proof}

\begin{proof}[Proof of Proposition~\ref{pro:propSclass}]
Item $(1)$ follows from Proposition \ref{pro:ConformalChange}  since $\d (\d u\circ \I)=F_g-F_h$. Item $(2)$ follows from the change of variable formula.

For item $(3)$, write $h=e^{2u}g$ and $k=e^{2v}h$. Then by item $(1)$ we have
\begin{align*}
\cS(g,h)+\cS(h,k)-\cS(g,k) & = \frac{1}{4}\int_\bA\left( u(F_g+F_h)+v(F_h+F_k)-(u+v)(F_g+F_k)\right) \\
& = \frac{1}{4}\int_\bA \left(u(F_h-F_k) - v(F_g-F_h)\right) \\
& = \frac{1}{4}\int_\bA \left(u\ \d (\d v \circ \I) - v\ \d (\d u\circ \I)\right)~.
\end{align*}
Recall that $\bA=\Proj(V)\times\Proj(V)\setminus\Delta$ and parametrize $\Proj(V)$ by the circle $\mathbb R/\mathbb Z$. 

Let $\{C_N\}_{N\in\mathbb N}$ be a sequence of polygonal curves in $\bA$ that converges uniformly to $\Delta$, and denote by $V_N=\{\alpha_1^N,...,\alpha_{2k_N}^N\}$ the (cyclically oriented) vertices of $C_N$. By construction 

$$\lim_{N\to\infty}\sup_{x\in V_N}\max\{\vert u(x)\vert,\vert v(x)\vert\}=0.$$
Let $\bA_N$ be the bounded connected component of $\bA\setminus C_N$. Since $\d u\wedge \d v\circ \I = \d v\wedge \d u\circ \I$, applying Stokes to $\bA_N$ gives
\[\cS(g,h)+\cS(h,k)-\cS(g,k)=\frac{1}{4}\ \underset{N \to \infty}{\lim} \int_{C_N} (u\d v\circ \I - v\d u\circ \I)~.\]
Since each side of $C_N$ is lightlike, $\I$ acts by $\pm 1$ on its tangent direction. Hence
\begin{align*}
\left\vert\int_{C_N} (u\d v\circ \I - v\d u\circ \I)\right\vert
&\leq 
\sup_{x\in V_N}\vert u(x)\vert\,\VB(v,C_N)
+\sup_{x\in V_N}\vert v(x)\vert\,\VB(u,C_N)\\
&\leq
\sup_{x\in V_N}\max\{\vert u(x)\vert,\vert v(x)\vert\}
\bigl(\VB(u,C_N)+\VB(v,C_N)\bigr).
\end{align*}

The result now follows from  Corollary~\ref{cor:BV}. \end{proof}

\subsection{The Liouville action between two uniformizing metrics}

A {\em uniformization} of the conformal annulus $\bA$ is a $C^3$ split diffeomorphism between $\bA$ and $\dS_+$. The {\em uniformizing metric} is then the pullback of the de Sitter metric by a uniformization. In particular, any uniformizing metric is of the form $\Phi^*g_0$ where $g_0$ is the "standard" de Sitter metric and $\Phi$ is a $C^2$ conformal diffeomorphism of $\bA$. Here we prove the following.

\begin{proposition}\label{pro:ActionBetweenUniformizing}
Any uniformizing metric $h$ on $\bA$ is in $\cS_0$ and $\cS(h,g_0)=0$.
\end{proposition}

\begin{proof}
Let us first prove that $h$ is in the $\cS$-class of $g_0$. We have $h=\Phi^* g_0$ for a $C^2$ conformal diffeomorphism of $\bA$. By Proposition \ref{pro:SplitDiff}, $\Phi$ is given in the splitting $\bA=\Rp\times \Rp\setminus \Delta$ by $\Phi(x,y)=(\varphi(x),\varphi(y))$ where $\varphi$ is an orientation-preserving $C^2$-diffeomorphism of $\Rp$. In particular, identifying $\Rp$ with $\mathbb R\cup\{
\infty\}$ we have
\[h=\Phi^*\left(2\frac{\d x\ \d y}{(x-y)^2}\right)=2\frac{\varphi'(x)\varphi'(y)}{(\varphi(x)-\varphi(y))^2}\d x\ \d y~.\]
Thus, we can write $h=e^{2u}g_0$ for
\[u=\frac{1}{2}\log\left(\frac{(x-y)^2\varphi'(x)\varphi'(y)}{(\varphi(x)-\varphi(y))^2} \right)~.\]
 The following lemma follows immediately from \cite{Tyurin:1978aa}.
\begin{lemma}\label{lem:LinkWithSchwarzian}
Using the notation above, we have the following estimates \[u(x,y)= \frac{1}{12}(x-y)^2S_\varphi (x)+ o((x-y)^2),\]
where $S_\varphi$ is the Schwarzian derivative of $\varphi$. 
\end{lemma} 

\begin{proof}
   For $y=x+\epsilon$, the Taylor expansion of $\varphi$ and $\varphi'$ at $x$ gives the following expansion, as observed in \cite{Tyurin:1978aa}:
    \[\frac{\varphi'(x) \varphi'(y)}{(\varphi(x)-\varphi(y))^2} = \frac{1}{(x-y)^2} + \frac16 S_\varphi (x) + o (1).\]
    Plugging this into the expression of $u$, we obtain 
    \[u (x, y) = \frac12 \log \left(1 + \frac16 (x-y)^2 S_\varphi(x) + o ((x-y)^2)\right) =  \frac{1}{12}(x-y)^2S_\varphi (x)+ o((x-y)^2)\]
   as claimed. 
\end{proof}

The above lemma directly implies that $u$ tends to $0$ uniformly on $\partial\bA$ and thus $u$ belongs to 
$L^\infty(\bA)$. 

Observe that $K(g_0)=K(h)=1$, so by Proposition \ref{pro:ConformalChange}, $\dal_{g_0}u = e^{2u}-1$ is in $L^\infty(\bA)$. Moreover, $\dal_{g_0}u$ is in $\mathsf{L}^1(\bA,\d a_{g_0})$ by Lemma \ref{lem:LinkWithSchwarzian}. This completes the proof that $h$ is in $\cS_0$.

To prove that $\cS(h,g_0)=0$, observe that the split invariance implies that for a $1$-parameter subgroup $\{\Phi_t\}_{t\in \R}$ of split diffeomorphisms, we have
\[\cS(\Phi^*_{t+s} g_0,\Phi^*_s g_0) = \cS(\Phi^*_tg_0,g_0)~.\]
Hence, it suffices to prove that 
\[\left. \frac{\d}{\d t}\right\vert_{t=0} \cS(\Phi_t^*g_0,g_0) = 0\ ,\]
for any $1$-parameter subgroup. Let $u_t$ be such that   
$\Phi_t^* g_0=e^{2u_t}g_0$ and 
\[\alpha\defeq \left.\frac{\d }{\d t}\right\vert_{t=0}u_t~.\]
Let $\xi$ be the vector field that generates $\phi_t$, then we get 
$$
\alpha(x,y)=\frac{1}{2}\left(\xi'(x)+\xi'(y)\right) -\left(\frac{\xi(x)-\xi(y)}{x-y}\right)=O((x-y)^2)\ , 
$$
since $\xi$ is $C^3$. It follows that $\alpha $ is in $L^1(\d a_{g_0})$. In particular, we have 
\begin{align}
\left.\frac{\d}{\d t}\right\vert_{t=0} \cS(\Phi^*_tg_0,g_0) =  \int_\bA \alpha\  \d a_{g_0} ~.\label{eqref:Salapha}	
\end{align}
By Proposition \ref{pro:ConformalChange}, and since $g_0$ and $\Phi_t^*g_0$ have constant curvature $\kappa$
\[\d(\d u_t\circ \I)= F_{g_0}-F_{\Phi^*_tg_0} = \kappa(\d a_{g_0} -\d a_{\Phi^*_tg_0})=\kappa (1-e^{2u_t})\d a_{g_0}\ .\]By Lemma \ref{lem:LinkWithSchwarzian}, for every $t$, the function $u_t$ extends smoothly to $\Ein $ to a $C^2$-function vanishing on $\bA$. Thus, we have
\begin{equation}\label{eq:1}\int_{\bA}  \d \left(\d u_t\circ \I\right)=  0~.\end{equation}
It follows that 
\[
\int_\bA (1-e^{2u_t})\ \d a_{g_0}=0.\]
Taking the derivative at $t=0$ and using equation \eqref{eq:1} yields
\[\int_\bA \alpha\ \d a_{g_0}= 0~,\]
and the result follows by equation \eqref{eqref:Salapha}.
\end{proof}

\section{Positive curves}\label{sec:positive-curves}

\subsection{Crossratio}

Denote by $(\Rp)^{(4)}$ the space of $4$-tuples of pairwise distinct points in $\Rp$.

\begin{definition}
A {\em crossratio} is a continuous function $\bb$ from $(\Rp)^{(4)}$ to $\R$ that satisfies the cocycle relations
\begin{equation}\label{eq:Cocycle1}
\bb(x,w,X,Y)\bb(w,y,X,Y)=\bb(x,y,X,Y) \ ,
\end{equation}
\begin{equation}\label{eq:Cocycle2}
\bb(x,y,W,Y)\bb(x,y,X,W)=\bb(x,y,X,Y)~.
\end{equation}
We say that a crossratio is {\em positive} if it satisfies furthermore $\bb(x,y,X,Y)>1$ for any cyclically oriented $4$-tuple $(x,y,X,Y)$ in $(\Rp)^{(4)}$.
\end{definition}

Smooth positive crossratios are related to smooth metrics as follows. The {\em diamond} defined by a cyclically oriented $4$-tuple $(x,y,X,Y)$ in $(\Rp)^{(4)}$ is 
\[\delta(x,y,X,Y) \defeq \left\{ (u,v) \in [x,y]\times [X,Y]\right\} \subset \bA ~,\]
where we recall that $\bA= (\Rp\times\Rp)\setminus\Delta$.

Given a positive crossratio $\bb$, define the $\bb$-area of a diamond by
\[\text{Area}_\bb\left( \delta(x,y,X,Y)\right) = \log(\bb(x,y,X,Y))~.\]
One easily checks that the cocycle relations (\ref{eq:Cocycle1}) and (\ref{eq:Cocycle2}) are equivalent to the additivity of $\text{Area}_\bb$ (see Figure \ref{fig:crossratio}). Since a smooth metric is characterized by its area form, any smooth positive crossratio $\bb$ defines a unique {\em crossratio metric} $g_\bb$.

\begin{figure}[h] 
\begin{center}
\includegraphics[width=6cm]{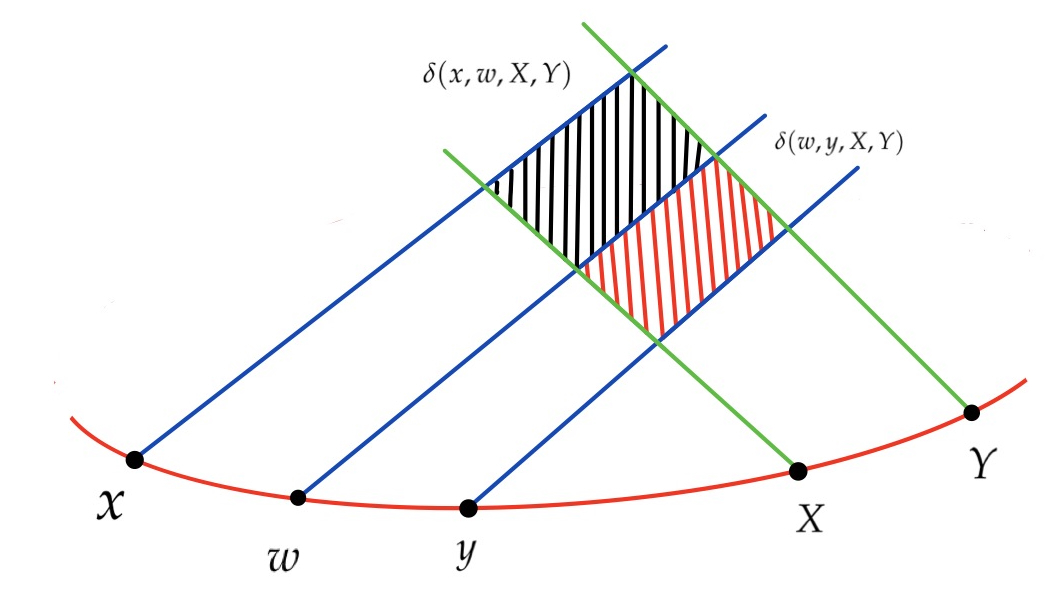}
\caption{Additivity of the crossratio area} 
\label{fig:crossratio}
\end{center}
\end{figure}

\subsection{Positive curves in flag varieties}\label{sec:positive-flag}

An important class of positive crossratios is obtained by considering positive curves in (self-dual) flag varieties (see \cite{Beyrer:2024aa}). Here, rather than providing an abstract definition of the theory, we highlight its key aspects and illustrate them with examples. The interested reader can consult \cite{Guichard:2025ab}.

Let $\G$ be a semisimple real Lie group with finite center and $\bF$ a flag variety of $\G$. The action of an element $g\in\G$ on $\bF$ is called {\em loxodromic} if there is a pair $(x_-,x_+)$ of $g$-invariant transverse points in $\bF$ such that if $U$ is the set of transverse flags to $x_-$, the sequence $\{g^n\}_{n\in\mathbb N}$ converges to the constant map $U \to \{x_+\}$ on any compact set in $U$. The points $x_+$ and $x_-$ are respectively called the attracting and repelling fixed points of $g$.

A {\em positive structure} on $\bF$ is a property of triples and quadruples of pairwise transverse flags that satisfy a set of axioms (see \cite{Guichard:2025aa} for more details). The set of pairs $(\G,\bF)$ admitting positive structures has been classified by Guichard--Wienhard in \cite{Guichard:2025ab} and is deeply connected with the theory of higher rank Teichm\"uller spaces.

As a consequence of the definition of positivity, we have for every transverse pair $p$ and $q$ in $\bF$ a finite family of nontrivial open cones $C^i_p(q)$ for $i$ in $\{1,\ldots,n\}$ in $\T_p\bF$ (which we identify with the Lie algebra of the unipotent radical of the stabilizer in $\G$ of $q$), such that for any $u$ in $C^i_p(q)$, the triple $(p,\exp(u),q)$ is positive. We call such a triple $(p,u,q)$ an {\em arrow}. By properties of positivity, if $(p,u,q)$ is an arrow, so is $(p,\lambda u,q)$. By positivity, the stabilizer of a positive triple is compact, and so is the stabilizer of an arrow.

This notion of positivity for $(\G,\bF)$ comes together with:
\begin{itemize}
	\item A notion of {\em positive curve}, which is a continuous map $\gamma$ from $\Rp$ to $\bF$ sending triples and quadruples of cyclically oriented pairwise distinct points in $\Rp$ to positive triples and quadruples of flags, respectively.
	\item A {\em conjugacy class of homomorphisms}  $\mathcal I$ from $\psld$ to $\G$, such that any semisimple element is mapped to a loxodromic element and $\iota(\psld)$ has compact centralizer for any $\iota$ in $\mathcal I$. 
	\item A notion of {\em circles} that are  $\iota$-equivariant positive curves $c$ from $\Rp$ to $\bF$ for $\iota$ in $\mathcal I$.
	\item For any dominant weight $\omega$ and positive curve $\gamma$, a positive crossratio $\bb_{\gamma,\omega}$ on $\gamma$; see \cite{Beyrer:2024aa}.
\end{itemize}

Fix $(\G,\bF)$ and a dominant weight $\omega$. By uniqueness of the crossratio on $\Rp$, there exists a positive integer $\lambda$ such that the positive crossratio associated with a circle has the form $$\bb_{c,\omega}(x,y,X,Y)=[x,y,X,Y]^\lambda\ ,$$
where $[\cdotp,\cdotp,\cdotp,\cdotp]$ is the standard projective crossratio on $\Rp$ normalized by
\[[0,1,x,\infty] = x~.\]

Thus, by construction, we can associate a metric $g_{\gamma,\omega}$ with any smooth positive curve $\gamma$ in $\bF$. When $\gamma$ is a circle, the metric is equal to $\lambda g_0$.

We define a {\em super-positive curve} to be a $C^1$-positive curve $f$ so that for any two distinct points $p$ and $q$ in $\Rp$, for any $u$ in $\T_p\Rp$, then $(f(p),\T_pf(u),f(q))$ is an arrow. We leave it to the reader to verify that the requirement that $f$ be positive is, in fact, redundant and that there exist $C^1$-positive curves that are not super-positive.

\begin{definition}[Liouville action for positive curves]
Let $(\G,\bF),~\omega$ and $\lambda$ be as above, and $\gamma$ a smooth positive curve in $\bF$ such that $g_{\gamma,\omega}$ is in the $\cS$-class of $\lambda g_0$. The {\em Liouville action of $\gamma$} is $$\cS(\gamma)\defeq\cS(\lambda g_0,g_{\gamma,\omega})\ .$$
\end{definition}

Observe that the Liouville action of curves is invariant under reparametrization. Indeed, if $\gamma$ is a smooth positive curve such that $g_{\gamma,\omega}$ is in the $\cS$-class of $\lambda g_0$, and $\varphi$ is a $C^3$ orientation-preserving diffeomorphism of $\Rp$, then 
\[g_{\delta,\omega} = \Phi^*g_{\gamma,\omega}~\]
for $\delta=\gamma\circ \varphi$ and $\Phi(x,y)=(\varphi(x),\varphi(y))$. In particular, by Chasles relation, we have
\[\cS(\lambda g_0,g_{\delta,\omega}) = \cS(\lambda g_0, \lambda \Phi^*g_0) + \cS(\lambda\Phi^*g_0, \Phi^*g_{\gamma,\omega}) = \cS(\lambda g_0,g_{\gamma,\omega})\]
where we used the split invariance and Proposition~\ref{pro:ActionBetweenUniformizing}. Similarly, the Liouville action is invariant under the left action of $\G$.

\subsection{Examples}

We now illustrate the notion of positivity on three different examples.

\subsubsection{The case $(\psld,\Rp)$}

Consider the pair $(\G,\bF)=(\psld,\Rp)$. In this setting, a positive triple is a triple of pairwise distinct points in $\Rp$ while a positive quadruple is a cyclically oriented 4-tuple of pairwise distinct points.

A positive curve is then an orientation-preserving homeomorphism, a super-positive one is a $C^1$-diffeomorphism, and a circle is just a M\"obius map.

In this situation, the dominant weight is unique and gives the standard cross-ratio. By invariance of the Liouville action under reparametrization, our invariant is always zero.

\subsubsection{The case $(\PO(2,2),\Ein )$}
Consider now $(\G,\bF)=(\PO(2,2),\Ein )$. A positive triple in $\Ein $ is a triple of points that spans a linear space of signature $(2,1)$. A positive quadruple is a $4$-tuple of points $(a,b,c,d)$ such that any subtriple is positive and whose projection on the first factor in $\Ein =\Rp\times\Rp$ is cyclically oriented.

Then a positive curve is of the form 
\[\function{
\gamma:}{\Rp}{\Rp\times \Rp\ ,}{t}{ (\psi(t),\phi(t))\ ,} \]
where $\varphi$ and $\psi$ are two orientation-preserving homeomorphisms. It is super-positive if and only if $\psi$ and $\phi$ are furthermore $C^1$-diffeomorphisms.

The group homomorphism $\iota$ is then the composition of the isomorphism between $\psld$ and $\SO_0(2,1)$ with the reducible embedding into $\PO(2,2)$. In particular, $\gamma$ is a circle when both $\phi$ and $\psi$ are M\"obius.

Let us now consider the Liouville action. By invariance under reparametrization, it is enough to consider the case $\psi=\Id$. There is a dominant weight $\omega$ such that the associated positive crossratio is
\[\bb_{\gamma,\omega}(x,y,X,Y) \defeq [x,y,X,Y]\cdotp [\varphi(x),\varphi(y),\varphi(X),\varphi(Y)]~.\]
When $\varphi$ is $C^3$, the crossratio metric is equal to $g_0+\Phi^*g_0$ for $\Phi(x,y)=(\varphi(x),\varphi(y))$ and so is in the $\cS$-class of $2g_0$ by Lemma \ref{lem:LinkWithSchwarzian}.

\subsubsection{The case $(\PSL_3(\R),\bF(\R^3))$}

Let $\bF(\R^3)$ be the space of full flags in $\R^3$ that we see as the set of pointed projective lines in $\mathbf{RP}^2$. A triple $(f_1,f_2,f_3)$ is positive if the points $(x_1,x_2,x_3)$ form a triangle inscribed in a triangle with edges $(\ell_1,\ell_2,\ell_3)$, where $f_i=(x_i,\ell_i)$. Similarly, a quadruple $(f_1,f_2,f_3,f_4)$ is positive if the points $(x_1,x_2,x_3,x_4)$ are the cyclically ordered vertices of a quadrilateral inscribed in a quadrilateral with edges $(\ell_1,\ell_2,\ell_3,\ell_4)$. See Figure \ref{fig:positivity}.

\begin{figure}[!h] 
\begin{center}
\includegraphics[width=15cm]{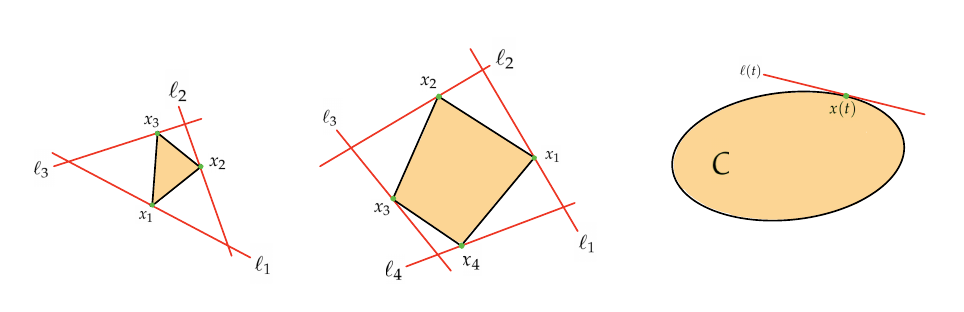}
\end{center}
\caption{From left to right: positive triple, quadruple, and curve} 
\label{fig:positivity}
\end{figure}

A positive curve in this setting is then a continuous map
\[\function{
\gamma:}{\Rp}{\bF(\R^3)\ ,}{t}{(x(t),\ell(t))\ ,} \]
where the curve defined by $x$ bounds a strictly convex set $C$ in $\mathbf{RP}^2$ and for any $t$ the line $\ell(t)$ is a support line of $C$ (see Figure \ref{fig:positivity}).

In this case, the group homomorphism $\iota$ is the composition of the isomorphism between $\psld$ and $\SO_0(2,1)$ with the embedding of $\SO_0(2,1)$ into $\PSL_3(\R)$. In particular, a circle in $\bF(\R^3)$ is the lift of a quadric in $\mathbf{RP}^2$ parametrized projectively. 

For the crossratio, there is a dominant weight $\omega$ such that
\[\bb_{\gamma,\omega}(t_1,t_2,s_1,s_2) = \frac{\langle \ell(s_1) \vert x(t_1)\rangle \langle \ell(s_2)\vert x(t_2)\rangle}{\langle \ell(s_1)\vert x(t_2)\rangle \langle \ell(s_2)\vert x(t_1)\rangle}~,\]
where to define $\langle \ell \vert x\rangle$ we choose a non-zero linear form on $\R^3$ with kernel $\ell$ and apply it to a chosen non-zero vector in $x$ (the above quotient is indeed independent of the choices). In particular, circles have Fuchsian crossratio of weight $2$.

\subsection{Piecewise circles and finiteness}

Given a conjugacy class of $\mathsf{SL}_2(\mathbb R)$ in $\G$, a {\em circle} (with respect to this class) is a closed orbit of an element of that conjugacy class $\psld$ in $\G$ isomorphic to $\Rp$.  This defines a family of circles on which $\G$ acts transitively. A {\em family of super-positive circles} is a family of positive curves such that 
\begin{enumerate}
	\item every circle in the family is  super-positive,
	\item there is a unique circle in the family passing through a given arrow.
\end{enumerate}

 Observe that there exist positive circles which are not super-positive: an example comes from the projection of the Veronese embedding (corresponding to the irreducible $\sld$ in $\mathsf{SO}(2,3)$) in the Einstein universe as in \cite[Section 5.3]{Collier:2019wu}. However, one can check that the examples in the previous paragraph are super-positive. Moreover, any positive circle invariant by the $\Theta$-positive subalgebra introduced in \cite[Section 7]{Guichard:2025ab} is super-positive. 

\begin{definition}\label{def:piecewise circle}
Let $(\G,\bF)$ be a positive flag variety. A {\em piecewise circle} is a $C^1$ positive curve $\gamma$ from $\Rp$ to $\bF$ that is piecewise a super-positive circle map. The endpoints of the intervals on which $\gamma$ is a circle map are called {\em turning points}. 
\end{definition}

\begin{theorem}\label{theo:Sclass}
Let $(\G,\bF)$ be a positive flag variety and $\omega$ a dominant weight such that circles have Fuchsian crossratio of weight $\lambda$. Then for any piecewise circle $\gamma$, the crossratio metric $g_{\gamma,\omega}$ is in the $\cS$-class of $\lambda g_0$. In particular, $\cS(\gamma)$ is finite.
\end{theorem}
\subsubsection{Preliminaries on super-positive circles}

Let $\alpha$ be the parameterization by $\Rp$ of a super-positive circle $C_0$ with $\alpha(0)=p$ and $\alpha(\infty)=q$. 
Let  $\iota$ be  the embedding of $\psld$ in $\G$ associated with $\alpha$ and $H\defeq\iota(h)$, where $h$ is a diagonal element in $\psld$ with attracting fixed point $0$ in $\Rp$.  
Recall that $\iota$ has a compact centralizer denoted by $\mathsf K$. Thus, if $\bC$ denotes the space of (unparametrized) circles in $\bF$ of the same type as $C_0$, we have
\[\bC = \G/\left(\iota(\psld)\times \mathsf K \right)~.\]
Set $\ell=\T_pC_0$ and define
\[\bC(p,\ell) = \{ D\in \bC~\vert ~ p\in D \text{ and }\T_pD=\ell\}~.\]

We will need several lemmas in the sequel.
\begin{lemma}[\sc Compactness]\label{lem:compact}
	Let $(p_0,u_0,q_0)$ be an arrow. Then there is a compact neighborhood  $K$ of $q_0$ such that
	\begin{itemize}
		\item for any $q$ in $K$, $(p_0,u_0,q)$ is an arrow.
		\item the set of circles in $\bC(p,\ell)$ that intersect $K$ is compact.
	\end{itemize} 
\end{lemma}
\begin{proof} Observe that being positive for triples is open. Thus, for $K$ a small enough neighborhood of $q$, the first item holds.
	Let, as before, $\G_1$ be the stabilizer in $\G$ of $(p,u)$. Observe that $\G_1$ acts algebraically in the algebraic variety $\bF$.
	By a corollary of Rosenlicht (Theorem \cite[Corollary 2.2.a]{Gromov:1988aa}), the orbits of $\G_1$ are embedded. 
	In particular, $K\cap \G_1\cdot q$ is compact for $K$ small enough. Since  $\bC(p,\ell)=\G_1/\mK_1$, for some compact subgroup $\mK_1$, the second item is satisfied.
\end{proof}

\begin{lemma}[\sc Contraction]\label{lem:Attracting} The action of $H$ on $\bC$ preserves $\bC(p,\ell)$. Moreover,
the action of $H^{-1}$ on $\bC(p,\ell)$ has the circle  $C_0$ as an attracting fixed point and more precisely
$$
\Vert \T_{C_0} H^{-1}\Vert <1\ .
$$
\end{lemma}
\begin{proof} 
The first statement follows from the definition of $\bC(p,\ell)$.
The element $H^{-1}$ is loxodromic with an attracting fixed point $q$ in $C_0$ and repelling fixed point $p$. Since $\bF$ is equal to its opposite, the attraction basin of $q$ is the set of elements of $\bF$ that are transverse to $p$. 

Let $C$ be an element of $\bC(p,\ell)$. Since $C$ is positive,  given $x$ in $C\setminus p$, $x$ is transverse to $p$. 

It follows that for any such $x$, $\{ H^{-n}x\}_{n\in \mathbb N}$ converges to $q$. 
Then by Lemma \ref{lem:compact}, we have that $\{H^{-n} D\}_{n\in \mathbb N}$ converges to a circle in $\bC(p,\ell)$ passing through $q$ and $p$ and tangent to $\ell$, which thus must be equal to $C_0$ by the second item in the definition of super-positive circles.

Let $\G_1$ be the subgroup of $\G$ preserving $(p,\ell)$,  $\ms L_0$ the centralizer in $\G$ of $H$, $\ms L_1$ the centralizer in $\G_1$ of $H$. Then  $(p,\ell)$ is fixed by $\ms L_1$ since $H$ has a unique fixed point in $\bC(p,\ell)$ near $(p,\ell)$ by the previous discussion.

Observe now that $\G_1$ acts transitively on $\bC(p,\ell)$ and therefore is a $\G_1$-space. Let $\ms H_{0}$ be the stabilizer of $C_0$. Then let  $\mk g$, $\mk g_1$ and $\mk h_{0}$ be the Lie algebras of $\G$, $\G_1$, and  $\ms H_{0}$ respectively. Recall that $\Ad(H)$ is real diagonalizable as an endomorphism of $\mk g$. Since $\Ad(H)$ preserves both $\mk g_1$ and $\mk h_{0}$, it follows that we can write $\mk g_1=\mk h_{0}\oplus V$, where $V$ is stable by $\Ad(H)$. By the first paragraph, all the eigenvalues of $\Ad(H)$  on $V$ are no greater than $1$. Moreover, by the previous paragraph, $\Ad(H)$ does not have $1$ as an eigenvalue on $V$. It follows that all the eigenvalues of $\Ad(H)$ on $V$ are less than $1$ and thus $\Ad(H)$ is a contracting endomorphism on $\bC(p,\ell)$. \end{proof}

We observe that as a corollary, $\bC(p,\ell)$ is contractible. We will also need
\begin{lemma}\label{lem:Hht} There is an  $H$-invariant submanifold $\Sigma$ containing $q$ and transverse to $C_0$.
\end{lemma}
\begin{proof} We choose a representation $\rho$ of $\G$ in $\mathsf{SL}_N(\R)$, such that there exists an equivariant embedding of $\bF$ in $\mathbf P(\mathbb R^N)$. 

Let $\iota$ be the corresponding  representation of $\sld$ in  $\mathsf{SL}_N(R)$, let $f$ be a $\iota$-equivariant map from $\Rp$ to $\mathsf{SL}_N(R)$, $h$ be a generator of the diagonal group in $\sld$ and $H\defeq \iota(h)$. 

Since $H$ is $\mathbb R$-split, there exists a hyperplane invariant by $H$ that does not contain $q=f(\infty)$. Sending this hyperplane to $\infty$, we have a linear chart of $\mathsf{SL}_N(R)$ in which $H$ is linear and $\mathbb R$-split. Then, since the tangent line $\ell_0$ to $f(\Rp)$ at $\infty$ is $H$-invariant, there exists a hyperplane $V$ through $0$, which is transverse to $\ell$ and $H$-invariant. Then, the foliation $\cF$ by affine hyperplanes parallel to $V$ satisfies the required properties.
This concludes the proof.
\end{proof}

\subsubsection{The conformal factor}

Let $\gamma: \Rp \to \bF$ be a piecewise circle and write $g_\gamma = e^{2u}(\lambda g_0)$ where $\lambda$ is the weight of a circle and $g_0$ is the de Sitter metric on $\Rp\times\Rp\setminus\Delta$. We now prove several lemmas. 
\begin{lemma}\label{lem:uSclass0} The function $u$ satisfies the following properties
\begin{enumerate}
	\item The function $u$ is $C^0$,
	\item  there are finitely many horizontal and vertical lines outside of which $u$ is $C^\infty$,
	\item $u$ is zero on a neighborhood of $\partial \mathbf A\setminus W$, where $W$ is the set of turning points,
	\item  the distribution $\dal_g u$ is locally bounded.
\end{enumerate}
\end{lemma}
\begin{proof} This is a consequence of the fact that $\gamma$ is $C^1$ and smooth outside the turning points. For the last statement, let $z$ be a point in $\bA$ and $g_{flat}$ a flat $(1,1)$ metric conformal to $g_0$ in the neighborhood of $z$. Since $u$ is piecewise $C^2$, it follows that $\dal_{g_{flat}}u$ is bounded in the neighborhood of $z$. Since 
$$
(\dal_{g_0}u) \omega_{g_0}=(\dal_{g_{flat}}u) \omega_{g_{flat}}\ ,
$$
the same can be said about $\dal_{g}u$. \end{proof}
\begin{corollary}\label{coro:VBu}
	There is a polygonal curve $P$ on which $\VB(u,P)$ is finite. 
\end{corollary}
\begin{proof} We choose a polygonal curve that is transverse to the horizontal lines and vertical lines on which $u$ ceases to be $C^\infty$. Then $u$ restricted to $P$ is continuous and piecewise $C^1$. It follows that $u$ is of bounded variation on $P$ and the result follows. 
	\end{proof}
\subsubsection{The function $u$ on the neighborhood of a turning point}
Thanks to Corollary~\ref{coro:VBu} and Lemma~\ref{lem:uSclass0}, the theorem reduces to the following lemma.
\begin{lemma}\label{lem:uSclass} We have
\begin{enumerate}
	\item The functions $u$ and $\dal_g(u)$ tend to $0$ uniformly on $\partial \bA$.
	\item The distribution 
$\dal_g(u)$ belongs to $L^\infty(\bA,\d a_g)$.\end{enumerate} 
\end{lemma}
In fact, it is enough to control the behavior of $u$ and $\dal_g u$ in the neighborhood of a turning point $w$. 
Since $\gamma$ is a piecewise $C^1$-circle, we can find a subdivision of  $\Rp$ into finitely many intervals
$$
\Rp=\bigcup_{i=0,\ldots, n}I_i\ \ \ , \ \ I_i=[a_i,a_{i+1}],
$$  
such that $\gamma_i\defeq \gamma\vert_{I_i}$ is a circle parameterization and we have
$$
\gamma_i(a_{i+1})=\gamma_{i+1}(a_{i+1})\ \ ,\ \  \dot\gamma_i(a_{i+1})=\dot\gamma_{i+1}(a_{i+1})\ .
$$
Since 
$$
K\defeq\bigcup_{\vert i-j\vert\geq 2}I_i\times I_j\ ,
$$
is compact in $\bA$, by Lemma \ref{lem:uSclass0}, $\dal_g u$ belongs to $L^1(K,\d a_{g})$ and  $L^\infty (K,\d a_{g})$. Moreover $u$ is zero on $I_i\times I_i\setminus \Delta$.
Thus, the lemma follows from  
\begin{lemma}
We have
\begin{enumerate}
		\item The function $u$ tends to $0$ uniformly as one converges to $a_i$,
	\item The distribution 
$\dal_g(u)$ belongs to $L^\infty(\bA_i,\d a_g)$ and $L^1(\bA_i,\d a_g)$, where $\bA_i=I_i\times I_{i+1}$.
\end{enumerate} 

\end{lemma}

\begin{proof}
Since the ordering is meaningless, we can focus on $i=0$.
After a Möbius change of parametrization of $\Rp$, we further restrict to the following situation: 
	\begin{enumerate}
	\item let $\gamma_0$ be the restriction, on an interval $[-c,0]$ with $c$ positive, of the parametrization $\alpha_0$ of a circle $C_0$,
	\item then let $q\defeq \alpha_0(\infty)$, let $h$ be a diagonal element in $\sld$, and let $H\defeq \iota_0(h)$ where $\iota_0$ is the embedding of $\sld$ associated with $\alpha_0$,
	\item let $\Sigma$ be the hypersurface of $\bF$, transverse to $\alpha_0$ and passing through $q$, invariant by $H$ (provided by Lemma \ref{lem:Hht}).
\end{enumerate}

Let $\mathcal O$ be the open neighborhood of $C_0$ in $\mathbf C(p,\ell)$ of circles that intersect $\Sigma$. We saw in Lemma \ref{lem:Attracting} that $H^{-1}$ acts on $\mathcal O$ and furthermore that $C_0$ is an attracting fixed point of $H^{-1}$ with 
$$
\Vert \T_{C_0}H^{-1}\Vert < 1\ .
$$
\vskip 0.2 truecm 
\noindent{\sc First step:} a function on  $\mathcal O$ 

We now parameterize each of the circles $C$ in $\mathcal O$ by $\alpha_C$ uniquely defined by 
$$
\alpha_C(0)=p\ ,\  \dot\alpha_C(0)=\dot\alpha_0(0)\ ,\  \alpha_C(\infty)\in\Sigma\ .
$$
From the uniqueness of the parametrization we have that 
\begin{align}
\alpha_{H^{-1}C}=H^{-1}\circ \alpha_C\circ  h\ .\label{eq:alphaC}	
\end{align}
For each $C$, we then have a $C^1$ (piecewise $C^2$) metric $g_C$ on $[-c,0]\times [0,\infty]\setminus\{(0,0)\}$ and a $C^1$ (piecewise $C^2$) function $u_C$ defined by 
$$
g_C=e^{2u_C}g_0\ .
$$
It then follows from equation \eqref{eq:alphaC}, that
\begin{align}
u_{H^{-1}C}=u_C\circ h\ . \label{eq:u_C}	
\end{align}
A final step in our construction is the choice once and for all $x$ and $w$ points in $\Rp$ with $x$ in $]-c,0[$ and $w$ in $]0,\infty[$. This allows us to define the $L$-shape
$$
L_k\defeq ([h^{k}(x),h^{k+1}(x)]\times [0,h^k(w)])\cup ([h^k(x),0]\times [h^{k+1}(w),h^{k}(w))\ ,
$$ 
and observe that \begin{align}
 L_k=h^k(L_0)\ . \label{eq:Lkh}	
 \end{align}

Observe that the intersection of $L_p$ and $L_k
$ has measure zero for $k$ different from $p$. Moreover
\begin{align}
\bigcup_{k=0}^\infty L_k=[x,0]\times[0,w]\setminus\{(0,0\}\eqdef A_0 .\label{eq:unioL}	
\end{align}

We now define the functions $U_L, V_L$ and $W_L$ on $\mathcal O$ by 
\begin{align}
	V_L(C)&\defeq\int_{L_0}\vert \dal_{g_C} u_C\vert  \ \omega_{g_C}\ ,\\
	U_L(C)\defeq\max_{L_0}(\vert u_C\vert)\ &, \
	W_L(C)\defeq\max_{L_0}(\vert \dal_{g_C} u_C\vert)\ .
	\end{align} 
\vskip 0.2 truecm
\noindent{\sc Second step:} We now prove that  the functions $V_L$,  $U_L$  and $W_L$ are Lipschitz on $\mathcal O$, and moreover,
\begin{align}
V_L(H^{-k}C)&=\int_{L_k}\vert \dal_{g_C} u_C\vert  \ \omega_{g_C}\ \ , \\ \ U_L(H^{-k}C)=\max_{L_k}(\vert u_C\vert)\ &, \ \ W_L(H^{-k}C)=\max_{L_k}(\vert \dal_{g_C} u_C\vert)\ .\label{eq:UWk}
\end{align}
The fact that the functions are Lipschitz follows from the fact that for any compact $K$ in $[-c,0]\times[0,\infty]$, $C\mapsto u_C$ is a smooth function with values in $C^\infty(K,\mathbb R)$. Let us check the final statement. By definition, setting $C_k=H^{-k}(C)$
\begin{align}
V_L(C_k)&=\int_{L_0}\vert\dal_{g_{C_k}} u_{C_k}\vert\cdot \omega_{g_{C_k}}=\int_{L_0}\vert \d (\d u_{C_k}\circ I)\vert \\ &=\int_{L_0}\vert\d (\d u_{C}\circ h^k\circ I)\vert=\int_{L_0}\vert\d (\d u_{C}\circ I\circ h^k)\vert=\int_{L_k}\vert\d (\d u_{C}\circ I)\vert\ .	
\end{align}
Here we used equation \eqref{eq:u_C} in the third equality and equation \eqref{eq:Lkh} in the last.  The assertion  \eqref{eq:UWk} that $U_L(H^{-k}C)=\max_{L_k}(\vert u_C\vert)$  follows from a similar proof. This concludes the proof of the second step.

\vskip 0.2 truecm
\noindent{\sc Final  step:} It follows from the previous step that $U_L(C_k)$ converges to $U_L(C_0)=0$ when $k$ goes to infinity. This implies that $\max_{L_k}\vert u\vert$ converges to zero when $k$ converges to infinity, and thus, $u_C$ converges uniformly to zero as one approaches 0. The same holds for $\max_{L_k}\vert \dal_{g_C}u_C\vert$. 

Moreover, by equation \eqref{eq:unioL}
$$
\int_{A_0}\vert \d (\d u_C\circ I)\vert =\sum_{k=0}^{\infty}\int_{L_k}\vert \d (\d u_C\circ I)\vert=\sum_{k=0}^{\infty}V_L(H^{-k}(C))\ .
$$

But for $C$ in a compact,  $\vert V_L(C_0)-V_L(C)\vert \leq K_1 d(C_0,C)$ since $V$ is Lipschitz.  It follows that for $k$ large enough, there exists a constant $K_2$ such that  
$$
\vert V(C_k)\vert \leq K_2\lambda^k\ ,
$$
since $V_L(C_0)=0$ and $\Vert \T_{C_0}H^{-1}\Vert<\lambda<1$. The result follows.
\end{proof}
\vskip 1 truecm
\begin{center}
{\sc Data availability statement:}\\	
There is no data associated to this manuscript.\\
\vskip 0.3 truecm
{\sc Conflict of interest statement:}\\
All authors declare that they have no conflicts of interest.
\end{center}
\bibliographystyle{amsplain}
\providecommand{\bysame}{\leavevmode\hbox to3em{\hrulefill}\thinspace}
\providecommand{\MR}{\relax\ifhmode\unskip\space\fi MR }
\providecommand{\MRhref}[2]{%
  \href{http://www.ams.org/mathscinet-getitem?mr=#1}{#2}
}
\providecommand{\href}[2]{#2}

\end{document}